\documentclass[english,12pt]{amsart}
\usepackage{amssymb}
\usepackage{amsfonts}
\usepackage{amscd}
\usepackage[all]{xypic}
\usepackage{color}
\usepackage{enumerate}
\usepackage[T1]{fontenc}
\RequirePackage[colorlinks=true, breaklinks=true, urlcolor= black, linkcolor= black, citecolor= black, bookmarksopen=true,linktocpage=true,plainpages=false,pdfpagelabels,unicode]{hyperref}

\CompileMatrices

\swapnumbers

\def\deg{{\rm deg}}

\def\11{{\mathbf 1}}
\def\AA{{\mathbb A}}

\def\CC{{\mathbb C}}

\def\FF{{\mathbb F}}
\def\GG{{\mathbb G}}

\def\PP{{\mathbb P}}
\def\QQ{{\mathbb Q}}

\def\ZZ{{\mathbb Z}}

\def\cB{{\mathcal B}}
\def\cC{{\mathcal C}}

\def\cP{{\mathcal P}}

\def\cX{{\mathcal X}}

\mathchardef\alphag="7C0B \mathchardef\betag="7C0C
\mathchardef\gammag="7C0D \mathchardef\deltag="7C0E
\mathchardef\varepsilong="7C22 \mathchardef\varphig="7C27
\mathchardef\psig="7C20 \mathchardef\zetag="7C10
\mathchardef\epsilong="7C0F \mathchardef\rhog="7C1A
\mathchardef\taug="7C1C \mathchardef\upsilong="7C1D
\mathchardef\iotag="7C13 \mathchardef\thetag="7C12
\mathchardef\pig="7C19 \mathchardef\sigmag="7C1B
\mathchardef\etag="7C11 \mathchardef\omegag="7C21
\mathchardef\kappag="7C14 \mathchardef\lambdag="7C15
\mathchardef\mug="7C16 \mathchardef\xig="7C18
\mathchardef\chig="7C1F \mathchardef\nug="7C17
\mathchardef\varthetag="7C23 \mathchardef\varpig="7C24
\mathchardef\varrhog="7C25 \mathchardef\varsigmag="7C26
\mathchardef\Omegag="7C0A \mathchardef\Thetag="7C02
\mathchardef\Sigmag="7C06 \mathchardef\Deltag="7C01
\mathchardef\Phig="7C08 \mathchardef\Gammag="7C00
\mathchardef\Psig="7C09 \mathchardef\Lambdag="7C03
\mathchardef\Xig="7C04 \mathchardef\Pig="7C05
\mathchardef\Upsilong="7C07

\newcounter{theoremcntr}[subsection]
\renewcommand*{\thetheoremcntr}{%
  \ifnum\value{subsection}=0 %
    \thesection
  \else
    \thesubsection
  \fi
  .\arabic{theoremcntr}%
}
\numberwithin{equation}{subsection}
\renewcommand*{\theequation}{%
  \ifnum\value{subsection}=0 %
    \thesection
  \else
    \thesubsection
  \fi
  .\arabic{equation}%
}

\newtheorem{thm}[theoremcntr]{Theorem}
\newtheorem{lem}[theoremcntr]{Lemma}
\newtheorem{cor}[theoremcntr]{Corollary}
\newtheorem{prop}[theoremcntr]{Proposition}

\newtheorem{maintheorem}%[subsection]
{Theorem}
\newtheorem{mainprop}[maintheorem]{Proposition}

\theoremstyle{definition}
\newtheorem{definition}[theoremcntr]{Definition}

\newtheorem{def-prop}[theoremcntr]{Proposition-Definition}
\newtheorem{def-theorem}[theoremcntr]{Theorem-Definition}
\newtheorem{def-lem}[theoremcntr]{Lemma-Definition}

\theoremstyle{remark}
\newtheorem{remark}[theoremcntr]{Remark}

\theoremstyle{plain}

\def\boxit#1#2{\setbox1=\hbox{\kern#1{#2}\kern#1}%
\dimen1=\ht1 \advance\dimen1 by #1 \dimen2=\dp1 \advance\dimen2 by
#1
\setbox1=\hbox{\vrule height\dimen1 depth\dimen2\box1\vrule}%
\setbox1=\vbox{\hrule\box1\hrule}%
\advance\dimen1 by .4pt \ht1=\dimen1 \advance\dimen2 by .4pt
\dp1=\dimen2 \box1\relax}

\mathchardef\alphag="7C0B \mathchardef\betag="7C0C
\mathchardef\gammag="7C0D \mathchardef\deltag="7C0E
\mathchardef\varepsilong="7C22 \mathchardef\varphig="7C27
\mathchardef\psig="7C20 \mathchardef\zetag="7C10
\mathchardef\epsilong="7C0F \mathchardef\rhog="7C1A
\mathchardef\taug="7C1C \mathchardef\upsilong="7C1D
\mathchardef\iotag="7C13 \mathchardef\thetag="7C12
\mathchardef\pig="7C19 \mathchardef\sigmag="7C1B
\mathchardef\etag="7C11 \mathchardef\omegag="7C21
\mathchardef\kappag="7C14 \mathchardef\lambdag="7C15
\mathchardef\mug="7C16 \mathchardef\xig="7C18
\mathchardef\chig="7C1F \mathchardef\nug="7C17
\mathchardef\varthetag="7C23 \mathchardef\varpig="7C24
\mathchardef\varrhog="7C25 \mathchardef\varsigmag="7C26
\mathchardef\Omegag="7C0A \mathchardef\Thetag="7C02
\mathchardef\Sigmag="7C06 \mathchardef\Deltag="7C01
\mathchardef\Phig="7C08 \mathchardef\Gammag="7C00
\mathchardef\Psig="7C09 \mathchardef\Lambdag="7C03
\mathchardef\Xig="7C04 \mathchardef\Pig="7C05
\mathchardef\Upsilong="7C07

 \DeclareMathOperator*{\Spec}{Spec}

\definecolor{orange}{rgb}{1,0.5,0}

\thanks{The authors would like to thank Gal Binyamini, Jonathan Pila, Arne Smeets, Jan Tuitman, and Alex Wilkie for interesting discussions on the topics of the paper, and Per Salberger for sharing his preprint with us and for interesting exchanges of ideas. The authors W.C.,  R.C., and K.H.N.\ were partially supported by the European Research Council under the European Community's Seventh Framework Programme (FP7/2007-2013) with ERC Grant Agreement nr. 615722 MOTMELSUM and thank the Labex CEMPI  (ANR-11-LABX-0007-01). The author W.C.\ is affiliated on a voluntary basis with the research group imec-COSIC at KU Leuven and with the Department of Mathematics: Algebra and Geometry at Ghent University. The authors R.C.\ and P.D.\ are partially supported by KU Leuven IF C14/17/083. The author K.H.N.\ is partially supported by the Fund for Scientific Research - Flanders (Belgium) (FWO) 12X3519N}

\title[The dimension growth conjecture, polynomial in the degree]{The dimension growth conjecture, polynomial in the degree and without logarithmic factors}

\author[Castryck]{Wouter Castryck}
\address{KU Leuven, Department of Mathematics,
B-3001 Leu\-ven, Bel\-gium}
\email{Wouter.Castryck@kuleuven.be}
\urladdr{http://homes.esat.kuleuven.be/~wcastryc/}

\author[Cluckers]{Raf Cluckers}
\address{CNRS, Univ.~Lille,  UMR 8524 - Laboratoire Paul Painlev\'e, F-59000 Lille, France, and,
KU Leuven, Department of Mathematics, B-3001 Leu\-ven, Bel\-gium}
\email{Raf.Cluckers@univ-lille.fr}
\urladdr{http://rcluckers.perso.math.cnrs.fr/}

\author[Dittmann]{Philip Dittmann}
\address{KU Leuven, Department of Mathematics,
 B-3001 Leu\-ven, Bel\-gium}
\email{philip.dittmann@kuleuven.be}

\author[Nguyen]{Kien Huu Nguyen}
\address{KU Leuven, Department of Mathematics,
 B-3001 Leu\-ven, Bel\-gium}
\email{kien.nguyenhuu@kuleuven.be}


\begin{document}

\begin{abstract}
We address Heath-Brown's and Serre's dimension growth conjecture (proved by Salberger), when the degree $d$ grows. Recall that Salberger's dimension growth results give bounds of the form $O_{X, \varepsilon} (B^{\dim X+\varepsilon})$  for the number of rational points of height at most $B$ on any integral subvariety $X$ of $\PP^n_\QQ$ of degree $d\geq 2$, where one can write $O_{d,n, \varepsilon}$ instead of $O_{X, \varepsilon}$ as soon as $d\geq 4$.
Our main contribution is to remove
the factor $B^\varepsilon$ as soon as $d \geq 5$, without introducing a factor $\log B$, while moreover obtaining polynomial dependence on $d$ of the implied constant.
Working polynomially in $d$ allows us to give a self-contained and slightly simplified treatment of dimension growth for degree $d \geq 16$, while in the range $5 \leq d \leq 15$ we invoke results by Browning, Heath-Brown and Salberger. Along the way we improve the well-known bounds due to Bombieri and Pila on the number of integral points of bounded height on affine curves and those by Walsh on the number of rational points of bounded height on projective curves. The former improvement leads to a slight sharpening of a recent estimate due to Bhargava, Shankar, Taniguchi, Thorne, Tsimerman and Zhao on the size of the $2$-torsion subgroup of the class group of a degree $d$ number field.
Our treatment builds on recent work by Salberger which brings in many primes in Heath-Brown's variant of the determinant method, and on recent work by Walsh and Ellenberg--Venkatesh, who bring in the size of the defining polynomial. We also obtain lower bounds showing that one cannot do better than polynomial dependence on $d$.
\end{abstract}

\maketitle

\section{Introduction and main results}

\subsection{}
Following a question raised by Heath-Brown in the case of hypersurfaces~\cite[p.~227]{Heath-Brown-cubic},
Serre formulated twice a question about rational points on a projective variety $X$ of degree $d$, see \cite[p.~27]{Serre-Galois} and~\cite[p.~178]{Serre-Mordell}, which was dubbed the dimension growth conjecture by Browning in~\cite{Browning-Q}. The question puts forward concrete upper bounds on the number of such points with height at most $B$, as a function of $B$. This dimension growth conjecture is now a theorem due to Salberger \cite{Salberger-dgc} (and others under various conditions on $d$);  moreover, for $d\geq 4$ Salberger obtains complete uniformity in $X$, keeping only $d$ and the dimension of the ambient projective space fixed, thereby confirming a variant that had been proposed by Heath-Brown.

We remove from these bounds the factors of the form $B^\varepsilon$
when the degree $d$ is at least $5$, without creating a factor $\log B$, while
moreover obtaining polynomial dependence on $d$ of the constants. The approach with polynomial dependence in $d$ is implemented in all auxiliary results as well, and it in fact allows us to give a more direct and self-contained proof of the dimension growth conjecture for $d$ at least $16$; our treatment of the cases $d=5,\ldots,15$ is not self-contained and uses \cite{Brow-Heath-Salb} when $d>5$ and a result from \cite{Salberger-dgc} for $d=5$. In particular, we obtain similar improvements to bounds of Bombieri--Pila on the number of integral points of bounded height on affine irreducible curves \cite[Theorem 5]{Bombieri-Pila}, and by Walsh on the number of rational points of bounded height on integral projective varieties \cite[Theorems 1.1, 1.2, 1.3]{Walsh}.

The possibility of polynomial dependence on $d$ came to us via a question raised by Yomdin (see below Remark 3.8 of \cite{Burguet-Liao-Yang}) in combination with the determinant method with smooth parameterizations as in \cite{Pila:Mild},  refined in \cite{CPW}, and via the work by Binyamini and Novikov \cite[Theorem 6]{BN3}. The removal of  the factor $B^\varepsilon$ without needing  $\log B$ is recently achieved by Walsh \cite[Theorems 1.1, 1.2, 1.3]{Walsh} who combines ideas by Ellenberg and Venkatesh \cite{Ellenb-Venkatesh} with the determinant method based on $p$-adic approximation (rather than on smooth maps) due to  Heath-Brown \cite{Heath-Brown-Ann}, refined in \cite{Salberger-dgc}.

%Working polynomially in $d$ allows us to give a self-contained treatment of the dimension growth theorem in the case of large degree $d$, up to the use of basic algebraic geometry and number theory. Our improvements of bounds of Bombieri-Pila \cite[Theorem 5]{Bombieri-Pila} and by Walsh \cite[Theorems 1.1, 1.2, 1.3]{Walsh} work for all degrees $d$. %Let us note that the case of degree $d=3$

We point out a difference between the dimension growth conjecture and the context of Manin's conjecture \cite{Peyre}: the bounds in the former are valid for all heights $B\geq 1$ while for the latter, the asymptotics for large $B$ are studied. For further context we refer to \cite{Browning-Q}.

%raf: we should also cite Browning's book somewhere.
%raf: add arxiv number for bhargava?

\subsection{}
Let us make all this more precise.
We study the number
$$
N(X, B)
$$
of rational points of height at most $B$ on
subvarieties $X$ of $\PP^n$ defined over $\QQ$. Here, the %(projective)
height $H(x)$ of a $\QQ$-rational point $x$ in $\PP^n$ is given by
$$
H(x)=\max(|x_0|, \ldots,|x_n|)
$$
for an $n+1$-tuple $(x_0 ,\ldots , x_n )$ of integers $x_i$
which are homogeneous coordinates for $x$ and have greatest common divisor equal to $1$.
 %Correspondingly, we study the number $n(X; H)$ of
%rational points of height at most $H$ on
%subvarieties $X$ of $\AA^n$ defined over $\QQ$.
%The (affine) height $H(x)$ of $x$ in $\QQ^n$ is given by $\max ( |a_1|,|b_1|,\ldots, |a_n|,|b_n|  )$ when $x=(x_1,\ldots,x_n)$, $x_i=a_i/b_i$ with %$a_i$ and $b_i$ coprime and $b_i\not=0$ for each $i$ .

Salberger proves in \cite{Salberger-dgc} the so-called dimension growth conjecture raised as a question by Serre in \cite[p.~27]{Serre-Galois} following Heath-Brown's question \cite[p.~227]{Heath-Brown-cubic} :
% (previously raised by Heath-Brown \cite[p.~227]{Heath-Brown-cubic} in the case of hypersurfaces): %Heath-Brown reference correct here/make precise?
\begin{quote}\textbf{Dimension Growth \cite[Theorem 0.1]{Salberger-dgc} }
If $X$ is an integral projective variety of degree $d\geq 2$ defined over $\QQ$, then
$$
N(X,B) \leq O_{X,\varepsilon} (B^{\dim X + \varepsilon}).
$$
\end{quote}
%(Several values for $d$ had been obtained previously in \cite{Heath-Brown-Ann}, \cite{Brow-Heath-Salb}, \cite{Heath-Brown-Ann}.)
One should compare the bound for $N(X,B)$ from this theorem to the trivial upper bound $O_{d,n} (B^{\dim X + 1})$ that follows from Lemma \ref{lem:gen_Schwarz_Zippel} below. %\ref{lem:triv}.

%raf degree 2
A variant of this question from \cite[p.~178]{Serre-Mordell} replaces the factor $B^{\varepsilon}$ by $\log(B)^c$ for some $c$ depending on $X$, see Example \ref{ex:Serre} below.

Heath-Brown \cite{Heath-Brown-Ann} introduces a form of this conjecture with uniformity in $X$ for fixed $d$ and $n$, and he develops a new variant of the determinant method using $p$-adic approximation instead of smooth parameterizations as in \cite{Bombieri-Pila}, \cite{BN3},  \cite{Pila:Mild}, \cite{CPW}.
In \cite{Salberger-dgc}, Salberger proves this uniform version of the dimension growth conjecture %(raised in \cite{Brow-Heath-Salb}) when $d$ is different from $2$ and from $3$.
for $d\geq 4$.
\begin{quote}\textbf{Uniform Dimension Growth \cite[Theorem 0.3]{Salberger-dgc}}
For $X\subseteq \PP_\QQ^n$ an integral projective variety of degree $d\geq 4$, one has
$$
N(X,B) \leq O_{d,n,\varepsilon} (B^{\dim X + \varepsilon}).
$$
\end{quote}

Almost all situations of this uniform dimension growth had been obtained previously in \cite{Heath-Brown-Ann} and \cite{Brow-Heath-Salb}, including the case $d=2$ but without the (difficult) cases $d=4$ and $d=5$.  Our main contributions are to make the dependence on $d$ polynomial, to remove the factor $B^\varepsilon$ without having to use factors $\log B$, and to provide relatively self-contained proofs for large degree, with main result as follows.

\begin{maintheorem}[Uniform dimension growth]\label{thm:dgcdegree}\label{thm:dcgdegree}
Given $n>1$, there exist constants $c=c(n)$ and $e=e(n)$, such that for all integral projective varieties $X\subseteq \PP_\QQ^n$ of degree $d \geq 5$ and all $B\geq 1$ one has
\begin{equation}\label{eq:dgc:dim X}
N(X,B) \leq c d^e B^{\dim X}.
\end{equation}
%\change{Should we express some bound for $d=3$ as well ?
%Furthermore,
%\change{
%and, if $d=3$, then we obtain ?
%\begin{equation*}\label{eq:dgc:dim X3}
%N(X,B) \leq c B^{\dim X + \frac{\sqrt{3}  - 2 }{3} }.
%\end{equation*}
%\raf{I believe we don't need to assume anything more on $X$, like Arithmetically Cohen-Macauley or so.}
\end{maintheorem}

In a way, one cannot do better than polynomial dependence on $d$, see the lower bounds from Proposition \ref{prop:W1.1:opt} and Section \ref{sec7} below.
%Moreover, we give an almost self-contained proof of Theorem \ref{thm:dcgdegree} when $d\geq 16$, apart from some basic results from algebraic geometry and number theory.

%\change{--- (In this paper we will prove theorem \ref{thm:dcgdegree} for hypersurfaces $X$ in $\PP^n$, and in the subsequent paper \cite{} we will render a typical projection argument polynomial in $d$ to cover the case of general $X$ in Theorem \ref{thm:dcgdegree}.) ---}

%Along our way, we improve previous bounds on the rational points on planar projective curves of Walsh, and, of planar affine curves of Bombieri-Pila and Pila, with the same benefits: no $B^\varepsilon$, no $\log B$, and polynomial dependence on the degree $d$, see Theorems \ref{thm:W1.1} and \ref{thm:Bombieri-Pila} below.

We heavily rework results and methods of Salberger, Walsh, Ellenberg--Venkatesh, Heath-Brown, and Browning, and use various explicit estimates for Hilbert functions, for certain universal Noether polynomials as in \cite{RuppertCrelle}, and for solutions of linear systems of equations over $\ZZ$ from \cite{Bombi-Vaal}.
%raf: make explicit references.

%More precisely, we get polynomial dependence on $d$, and prove that the dependence cannot be better than polynomial (although we don't claim to find the optimal exponent).

%Let us phrase our main results of the paper:

%Some of the key auxiliary results we show in this paper are about rational points on curves.

\subsection{Rational points on curves and surfaces}

Let us make precise some of our improvements for counting points on curves and surfaces, which are key to Theorem \ref{thm:dcgdegree}.
We obtain the following improvement of Walsh's Theorem 1.1 \cite{Walsh}.

\begin{maintheorem}[Projective curves]\label{thm:W1.1}
Given $n>1$, there exists a  constant $c=c(n)$ such that for all $d>0$ and all integral projective curves $X\subseteq \PP_\QQ^n$ of degree $d$ and all $B\geq 1$ one has
$$
N(X,B) \leq c d^4 B^{2/d}.
$$
%Moreover, if $n=2$, then one can take $e=4$.
\end{maintheorem}

In view of Proposition \ref{prop:W1.1:opt} below, the exponent $4$ of $d$ in Theorem \ref{thm:W1.1} can perhaps be lowered, but cannot become lower than $2$ in general.
Several adaptations of results and proofs of \cite{Walsh} are key to our treatment and are developed in Section~\ref{sec3}.

%For $x\in \QQ$ with $x=a/b$ where $a$ and $b\not=0$ are coprime, let $H(x)$ be $\max(|a|,|b|)$, and for  $x\in \QQ^n$, let $H(x)$ be $\max_{i=1}^n H(x_i)$.

For affine counting we use the following notation for a variety $X\subseteq \AA_\QQ^n$ and a polynomial $f$ in $\ZZ[y_1,\ldots,y_n]$:
$$
N_{\rm aff}(X,B) := \#\{x\in \ZZ^n \mid |x_i|\leq  B\mbox{ for each $i$ and } x\in X(\QQ)\} ,
$$
and
$$
N_{\rm aff}(f,B) := \#\{x\in \ZZ^n\mid |x_i|\leq  B \mbox{ for each $i$ and }f(x)=0\}.
$$

By a careful elaboration of the argument from \cite[Remark 2.3]{Ellenb-Venkatesh} and an explicit but otherwise classical projection argument, we find the following improvement of bounds by Bombieri--Pila \cite[Theorem 5]{Bombieri-Pila} and later sharpenings by Pila \cite{Pila-ast-1995}, \cite{Pila-density-curve}, Walkowiak \cite{Walkowiak}, Ellenberg--Venkatesh \cite[Remark 2.3]{Ellenb-Venkatesh},  Binyamini and Novikov \cite[Theorem 6]{BN3}, and others.

\begin{maintheorem}[Affine curves]\label{thm:Bombieri-Pila}
Given $n>1$, there exists a constant $c=c(n)$ such that for all $d>0$, all integral affine curves $X\subseteq \AA_\QQ^n$ of degree $d$, and all $B\geq 1$ one has
$$
N_{\rm aff} (X,B) %:= \#\{x\in \ZZ^2 \mid x\in X(\QQ)\mbox{ and } |x_i|< B\mbox{ for each }i\}
\leq c d^3 B^{1/d} (\log B + d).
$$
%\change{Moreover, if $n=2$, then one can take $e= 4$.}
\end{maintheorem}

 A variant of Theorem \ref{thm:Bombieri-Pila} is given in Section \ref{sec5}, where $\log B$ is absent and instead the size of the coefficients of the polynomial $f$ defining the affine planar curve comes in.
 %raf: give a more precise reference.

It is well-known that Theorems  \ref{thm:dcgdegree}, \ref{thm:W1.1}, \ref{thm:Bombieri-Pila} imply similar bounds for varieties defined and integral over $\overline{\QQ}$ (instead of $\QQ$), by intersecting with a Galois conjugate and using a trivial bound, see  Lemma \ref{lem:Gal}.
The following improves Theorem 0.4 of \cite{Salberger-dgc} and is key to Theorem \ref{thm:dcgdegree}. It can be seen as an affine form of the dimension growth theorem, for hypersurfaces.

\begin{maintheorem}[Affine hypersurfaces]\label{thm:0.4}
Given $n>2$, there exist constants $c=c(n)$ and $e=e(n)$, such that for all polynomials $f$ in $\ZZ[y_1,\ldots,y_n]$ whose homogeneous part of highest degree $h(f)$ is irreducible over $\overline\QQ$ and whose degree $d$ is at least $5$, one has
$$
N_{\rm aff}(f,B) %:= \#\{x\in \ZZ^n\mid |x_i|< B \mbox{ for each $i$ and }f(x)=0\}
\leq c d^e B^{n-2}.
$$
%\change{and, for $d=4$, one has
%$$
%n(f,B) %:= \#\{x\in \ZZ^n\mid |x_i|< B \mbox{ for each $i$ and }f(x)=0\}
%\leq c B^{n-2} \log B.
%$$
\end{maintheorem}
One should compare the bound from this theorem to the trivial upper bound $O_{d,n} (B^{n-1})$ from Lemma \ref{lem:gen_Schwarz_Zippel}.

\subsection{Example and a question}\label{ex:Serre}
In Serre's example  \cite[p.~178]{Serre-Mordell} of the degree $2$ surface in $\PP^3$ given by the equation $xy=zw$, the logarithmic factor $\log B$  cannot be dispensed with in the upper bound. Hence, (\ref{eq:dgc:dim X}) of Theorem \ref{thm:dcgdegree} cannot hold for $d=2$ in general. For $d=3$, the bound from (\ref{eq:dgc:dim X}) remains wide open since already uniformity in $X\subseteq \PP^n$ of degree $3$ is not known for the uniform dimension growth with $O_{d,n,\varepsilon}(B^{\dim X+\varepsilon})$ as upper bound (see \cite{Salberger-dgc} for subtleties when $d=3$). For $d=4$, one may investigate whether (\ref{eq:dgc:dim X}) of Theorem \ref{thm:dcgdegree} remains true, that is, without involving a factor $B^\varepsilon$ or $\log B$.

 %When $d=4$ we have a factor $\log B$ in Theorem \ref{thm:dcgdegree}, but we don't know whether it is necessary.

\subsection{Lower bounds}

In Section \ref{sec7} we discuss the necessity of the polynomial dependence on $d$ in the above theorems.

%\begin{prop}\label{thm:W1.1:opt}
%For each integer $d>2$ there is an integral projective curve $X\subseteq \PP^2$ of degree $d$ and an integer $B\geq 1$ such that
%$$
%d B^{2/d} \leq N(X,B).
%$$
%\end{prop}

%\begin{proof} Fix $d>2$.
%Take a nonzero polynomial $f(x)$ of degree $d$ in one variable $x$ which vanishes at the integers $0,1,\ldots,d-1$.
%Then the projective curve $X$ in $\PP^2$ given by the equation $z^{d-1}y=f(x)$ is as desired with $B=d$. (Give details; optimize.)
%\end{proof}
%The lower bound of Proposition \ref{thm:W1.1:opt} should be compared with a question raised in \cite{Salberger-dgc} (between Theorems 0.12 and 0.13) about possible exponent $2$ or $2+\varepsilon$ for curves and consequences of this.

\begin{mainprop}\label{prop:W1.1:opt}
For each integer $d>0$ there is an integral projective curve $X\subseteq \PP^2$ of degree $d$ and an integer $B\geq 1$ such that
$$
\frac{1}{5}d^2 B^{2/d}  \leq N(X,B).
$$
In particular, in the statement of Theorem \ref{thm:W1.1} it is impossible to replace the factor $d^4$ with an expression in $d$ which is $o(d^2)$.
\end{mainprop}
Similarly we show that it is impossible to
replace the quartic dependence on $d$ of the bound from Theorem~\ref{thm:Bombieri-Pila} by a function in $o(d^2 / \log d)$. We also show that
in
Theorems~\ref{thm:dcgdegree} and~\ref{thm:0.4} we cannot take $e < 1$ resp.\ $e < 2$.
 %raf: ok for \ref{thm:Bombieri-Pila}  and for \ref{thm:0.4} ?

\subsection{An application}

When it comes to applications, our bounds can be used as substitutes for those by Salberger, Bombieri--Pila, and Walsh upon which they improve, potentially leading to stronger statements. A very recent example of such an application is Bhargava, Shankar, Taniguchi, Thorne, Tsimerman and Zhao's bound~\cite[Theorem~1.1]{BSTTTZ} on the number $h_2(K)$ of $2$-torsion elements in the class group of a degree $d > 2$ number field $K$, in terms of its discriminant $\Delta_K$. Precisely, they show that
\[ h_2(K) \leq O_{d,\varepsilon}(|\Delta_K|^{\frac{1}{2} - \frac{1}{2d} + \varepsilon}), \]
thereby obtaining a power saving over the trivial bound coming from the Brauer-Siegel theorem. This power saving is mainly accounted for by an application of Bombieri and Pila's bound from \cite[Theorem 5]{Bombieri-Pila}. In Section~\ref{sec5} we explain how our improved bound stated in Theorem~\ref{thm:Bombieri-Pila}, or rather its refinement stated in Corollary~\ref{cor:Ellenb}, allows for removal of the factor
$|\Delta_K|^\varepsilon$ as soon as $d$ is odd; if $d$ is even then we can replace it by $\log |\Delta_K|$.
\begin{maintheorem} \label{thm:bhargava_and_co}
For all degree $d > 2$ number fields $K$ we have
\[ h_2(K) \leq O_d( |\Delta_K|^{\frac{1}{2} - \frac{1}{2d}} (\log | \Delta_K|)^{1 - (d \bmod 2)} ) .\]
\end{maintheorem}
It is possible to make the hidden constant explicit, but targeting polynomial growth seems of lesser interest since $|\Delta_K|$ is itself bounded from below by an exponential expression in $d$, coming from Minkowski's bound.

\subsection{Structure of the paper}

In Section \ref{sec2} we render several results of Salberger \cite{SalbCrelle} explicit in terms of the degrees and dimensions involved.  In Section \ref{sec3}  we similarly adapt the results of Walsh \cite{Walsh}. Section \ref{sec5} completes the proofs of our main results in the hypersurface case, which is complemented by Section \ref{sec8}, in which we discuss projection arguments from \cite{Brow-Heath-Salb}, explicit in the degrees and dimensions, and thus finish the proofs of our main theorems. Finally, in Section \ref{sec7}, we provide lower bounds showing the necessity of polynomial dependence on $d$ in our main results. %and we give our application to counting certain number fields.

\section{The determinant method for hypersurfaces}\label{sec2}\label{sec:salberger}

With the aim of improving the results of \cite{Walsh} in the next section, we sharpen some results from Salberger's global determinant method. The main result of this section is Corollary \ref{cor:Salberger_determinant}, which improves on \cite[Lemmas 1.4, 1.5]{Salberger-dgc} (see also \cite[Theorem 2.2]{Walsh}).
This mainly depends on making \cite[Main Lemma 2.5]{SalbCrelle} in the case of hypersurfaces explicit in its independence of the degree.

Let $f$ be an absolutely irreducible homogeneous polynomial in $\ZZ[x_0,\dotsc, x_{n+1}]$ which is primitive, and let $X$ be the hypersurface in $\PP^{n+1}_\QQ$ defined by $f$. For $p$ a prime number, let $X_p$  denote the reduction of $X$ modulo $p$, i.e.\ the hypersurface in $\PP^{n+1}_{\FF_p}$ described by the reduction of $f \bmod p$.

\begin{lem}[Lemma 2.3 of \cite{SalbCrelle}, explicit for hypersurfaces] \label{lem:stalk}
Let $A$ be the stalk of $X_p$ at some $\FF_p$-point $P$ of multiplicity $\mu$ and let $\mathfrak{m}$ be the maximal ideal of $A$.
Let $g _{X,P}:\ZZ_{>0} \to\ZZ$
be the function given by $g_{X,P}(k) = \dim_{A/\mathfrak{m}} \mathfrak{m}^k / \mathfrak{m}^{k+1}$ for $k>0$. Then one has
\[
g(k) = \binom{n+k}{n} \mbox{ for } k < \mu
\]
and
\[
g(k) = \binom{n+k}{n} - \binom{n+k-\mu}{n} \mbox{ for } k\geq \mu.
\]
In particular, \[ g(k) \leq \frac{\mu k^{n-1}}{(n-1)!} + O_n(k^{n-2}) \] for all $k \geq 1$, where the implied constant depends only on $n$, as indicated.
\end{lem}
\begin{proof}
  The function $g$ is identical to the Hilbert function of the projectivized tangent cone of $X_p$ at $P$, which is a degree $\mu$ hypersurface in $\PP^{n}$.
  This gives the explicit expression for $g$, so it only remains to prove the estimate.

  Consider first $k < \mu$. Then \[ g(k) = \binom{n+k}{n} = \frac{k^{n}}{n!} + \frac{(n+1) k^{n-1}}{2(n-1)!} + O_n(k^{n-2}) .\]
  Since $\mu > k$, for $k\geq n$ we immediately obtain the desired inequality, and the $k$ between $1$ and $n$ are covered by choosing the constant large enough.

  Now consider $k \geq \mu$. Write $p(X)$ for the polynomial $\binom{n+X}{n}$ and $a_i$ for its coefficients.
  Then \[p(k) - p(k-\mu) = a_{n} (k^{n}-(k-\mu)^{n}) + a_{n-1}(k^{n-1}-(k-\mu)^{n-1}) + O_n(k^{n-2}) \]
  Observe that $a_{n} = 1/n!$, $a_{n-1} = (n+1)/(2(n-1)!) = a_n (n+1)n/2$,
  and write \[ k^{n}-(k-\mu)^{n} = \mu(k^{n-1} + (k-\mu)k^{n-2} + \dotsb + (k-\mu)^{n-1})\]
  as well as \[ k^{n-1}-(k-\mu)^{n-1} = \mu (k^{n-2} + \dotsb +(k-\mu)^{n-2}) .\]
  Considering $\mu \geq n(n+1)/2$, we have \[(k-\mu)^i k^{n-1-i} + \frac{(n+1)n}{2} (k-\mu)^{i-1} k^{n-1-i} \leq k^{n-1}\] for $i \geq 1$,
  and hence \[ a_{n} (k^{n} - (k-\mu)^{n}) + a_{n-1}(k^{n-1}-(k-\mu)^{n-1}) \leq \mu/n! (k^{n-1} + \dotsb k^{n-1}) = \mu k^{n-1}/(n-1)! \] as desired.
  The finitely many $\mu$ less than $n(n+1)/2$ are again taken care of by taking a sufficiently large constant in the $O_n$.
\end{proof}

%\begin{lem}
%  Assume $n=1$, so $X$ is a curve.
%  Consider $A$ as above, and let $(n_l(A))_{l \geq 1}$ be the non-decreasing sequence of integers $k \geq 0$ where $k$ occurs exactly $\dim_{A/m} m^k/m^{k+1}$ times. Write $A(s) = n_1(A) + \dotsb + n_s(A)$.
%  Then $A(s) \geq s^2/(2\mu) - 3/2 s$.
%\end{lem}
%\begin{proof}
%  By the last lemma, for any $k$ we have $\dim_{A/m} m^k/m^{k+1} \leq \mu$.
%
%  Let $(m_l)_{l \geq 1}$ be the non-decreasing sequence of integers $k \geq 0$ where $k$ occurs exactly $\mu$ times.
%  Then $m_l \leq n_l$ for all $l$, so it suffices to investigate the quantity $S(s) = m_1 + \dotsb + m_s$.
%
%  We have \[S(s) \geq \sum_{i=1}^{\lfloor s/\mu\rfloor} (i-1)\mu = \frac{\mu}{2} \lfloor s/\mu\rfloor \lfloor s/\mu-1\rfloor \geq  s^2/(2\mu) - 3/2 s, \] finishing the proof.
%\end{proof}

\begin{lem}
  Let $c, n, \mu > 0$ be integers.
  Let $g \colon \ZZ_{\geq 0} \to \ZZ_{\geq 0}$ be a function with $g(0) = 1$ and satisfying $g(k) \leq \frac{\mu k^{n-1}}{(n-1)!} + c k^{n-2}$ for $k > 0$.
  Let $(n_i)_{i \geq 1}$ be the non-decreasing sequence of integers $m \geq 0$ where $m$ occurs exactly $g(m)$ times. Then for any $s \geq 0$ we have
  \[ n_1 + \dotsb + n_s \geq (\frac{n!}{\mu})^{1/n} \frac{n}{n+1} s^{1+1/n} - O_{n,c}(s) .\]
\end{lem}
This statement is implicitly contained in the proof of \cite[Main Lemma 2.5]{SalbCrelle}, but we give the full proof to stress that the error term does not depend on $\mu$.
\begin{proof}
  Note that replacing $g$ by a function which is pointwise larger than $g$ at any point only strengthens the claim, so we may as well assume that
  \[ g(k) = \frac{\mu}{n!} (k^n - (k-1)^n) + c (k^{n-1} - (k-1)^{n-1}) \]
  for $k > 0$.
  Let $G \colon \ZZ_{\geq 0} \to \ZZ_{\geq 0}$ be given by $G(k) = g(0) + \dotsb + g(k) = \frac{\mu}{n!} k^n + c k^{n-1} + 1$.
  Now \[ (\frac{n!}{\mu})^{1/n} \frac{n}{n+1} G(k)^{1+1/n} = \frac{\mu k^{n+1}}{(n-1)!(n+1)} + O_{n,c}(k^n), \]
  and \[ 0 g(0) + \dotsb + k g(k) \geq \frac{\mu}{(n-1)!} \sum_{i \leq k} (i^n + O_n(c i^{n-1})) = \frac{\mu}{(n-1)!(n+1)} k^{n+1} + O_{n,c}(\mu k^n) .\]
  This proves the lemma for $s = G(k)$.

  To deduce the result for general $s > 0$, let $k$ be the unique integer with $G(k-1) < s \leq G(k)$, and use
  \begin{align*}
  n_1 + \dotsb + n_s \geq n_1 + \dotsb + n_{G(k)} - k g(k)
  &\geq (\frac{n!}{\mu})^{1/n} \frac{n}{n+1} G(k)^{1+1/n} + O_{n,c}(\mu k^n)  \\
  &\geq (\frac{n!}{\mu})^{1/n} \frac{n}{n+1} s^{1+1/n} + O_{n,c}(s)
  . \qedhere \end{align*}
\end{proof}

\begin{lem}
  Consider $A$ as in Lemma~\ref{lem:stalk}, and let $(n_i(A))_{i \geq 1}$ be the non-decreasing sequence of integers $m \geq 0$ where $m$ occurs exactly $\dim_{A/\mathfrak{m}} \mathfrak{m}^k/\mathfrak{m}^{k+1}$ times. Write $A(s) = n_1(A) + \dotsb + n_s(A)$.
  Then \[ A(s) \geq (n!/\mu)^{1/n}(n/(n+1)) s^{1+1/n} - O_n(s) ,\] where the implied constant only depends on $n$.
\end{lem}
\begin{proof}
This is immediate from the last two lemmas.
\end{proof}

As usual, write $\ZZ_{(p)}$ for the localization of $\ZZ$ at (the complement of) the prime ideal $(p)$.
\begin{lem}[Lemma 2.4 of \cite{SalbCrelle}, cited as in Appendix of \cite{Brow-Heath-Salb}]
  Let $R$ be a noetherian local ring containing $\ZZ_{(p)}$, $A=R/pR$, and consider ring homomorphisms $\psi_1, \dotsc, \psi_s \colon R \to \ZZ_{(p)}$.
  Let $r_1, \dotsc, r_s$ be elements of $R$. Then the determinant of the $s \times s$-matrix $(\psi_i(r_j))$ is divisible by $p^{A(s)}$.
\end{lem}

\begin{cor}[Main Lemma 2.5 of \cite{SalbCrelle}]
  Let $\cX \to \Spec \ZZ$ be the hypersurface in $\mathbb{P}^n_{\ZZ}$ cut out by the homogeneous polynomial $f$ as above, so $X$ is the generic fibre of $\cX$ and $X_p$ is the special fibre of $\cX$ over $p$.

  Let $P$ be an $\FF_p$-point of multiplicity $\mu$ on $X_p$ and let $\xi_1, \dotsc, \xi_s$ be $\ZZ$-points on $\cX$, given by some primitive integer tuples, with reduction $P$.
  Let $F_1, \dotsc, F_s$ be homogeneous polynomials in $x_0, \dotsc, x_n$ with integer coefficients.

  Then $\det\big(F_j(\xi_i)\big)$ is divisible by $p^e$ where 
  $$
  e\geq (n!/\mu)^{1/n}\frac{n}{n+1} s^{1+1/n} - O_n(s).
  $$
\end{cor}
\begin{proof}
  Let $P'$ be the image of $P$ under the closed embedding $X_p \hookrightarrow \cX$, and $R$ the stalk of $\cX$ at $P'$.
  Then $R$ is a noetherian local ring containing $\ZZ_{(p)}$, and $R/pR$ is the stalk of $X_p$ at $P$.
  Since $P'$ is a specialization of all the $\xi_i$ (this is precisely what it means that the $\xi_i$ have reduction $P$), it makes sense to evaluate an element of $R$ at each $\xi_i$, giving $s$ ring homomorphisms $R \to \ZZ_{(p)}$.

  The $F_i$ induce $\ZZ_{(p)}$-valued polynomial functions on an affine neighbourhood of $P'$, and hence give elements of $R$.
  The statement now follows from the preceding two lemmas.
\end{proof}

\begin{prop}\label{prop:Salberger_determinant}
  Let $\cX$ be as above. Let $\xi_1, \dotsc, \xi_s$ be $\ZZ$-points on $\cX$, and $F_1, \dotsc, F_s$ be homogeneous polynomials in $n+1$ variables with integer coefficients.
  Then the determinant $\Delta$ of the $s \times s$-matrix $(F_i(\xi_j))$ is divisible by $p^e$, where 
  $$
  e \geq (n!)^{1/n}\frac{n}{n+1} \frac{s^{1+1/n}}{ n_p^{1/n}} - O_n(s),
   $$
  and where $n_p$ is the number of $\FF_p$-points on $X_p$, counted with multiplicity.
\end{prop}
\begin{proof}
  This is identical to the proof of \cite[Lemma 1.4]{Salberger-dgc}, see also the appendix of \cite{Walsh} -- but we have eliminated the dependence on the constant on $d$.
\end{proof}

\begin{lem}
  In the situation above, if $p > 27d^4$ and $X_p$ is geometrically integral, i.e.~the defining polynomial $f$ has absolutely irreducible reduction modulo $p$, then $n_p \leq p^n + O_n(d^2 p^{n-1/2})$.
\end{lem}
\begin{proof}
  By \cite[Corollary 5.6]{CafureMatera} the number of $\FF_p$-points of $X_p$ counted without multiplicity is bounded by 
  \[ \frac{p^{n+1} + (d-1)(d-2)p^{n+1/2} + (5d^2+d+1) p^n - 1}{p-1} \leq p^n + O_n(d^2 p^{n-1/2}). \]  (This uses the lower bound on $p$ and the condition on $X_p$.)

  The singular points of $X_p$ all lie in the algebraic set cut out by $f$ and $\frac{\partial f}{\partial x_0}$, which can be assumed non-zero without loss of generality. This is an algebraic set all of whose components have codimension $2$ %by Krull's Height Theorem, 
  and the sum of the degrees of these components is bounded by $d^2$.
  The standard Lang--Weil estimate %(see the later Lemma \ref{lem:gen_Schwarz_Zippel})
   yields that there are $O_n(d^2 p^{n-1}) \leq O_n(dp^{n-1/2})$ points on this algebraic set and hence at most that many singular points, each of which has multiplicity at most $d$.
  Adding this term to the number of points counted without multiplicity yields the claim.
\end{proof}

\begin{lem}
  In the situation above, with $p > 27d^4$ and $X_p$ geometrically integral, we have $n_p^{1/n}/p - 1 \leq O_{n}(d^2 p^{-1/2})$.
\end{lem}
\begin{proof}
  Apply the general inequality $x^{1/n} -1 \leq x - 1$ for $x \geq 1$.
\end{proof}

We immediately obtain the following from Proposition \ref{prop:Salberger_determinant}.
\begin{cor}\label{cor:Salberger_determinant}
The determinant $\Delta$ from Proposition \ref{prop:Salberger_determinant} is divisible by $p^e$, where \[ e \geq (n!)^{1/n}\frac{n}{n+1} \frac{s^{1+1/n}}{p + O_{n}(d^2p^{1/2})} - O_n(s) .\]
\end{cor}
This is stated as Theorem 2.2 in \cite{Walsh}, but our statement is more precise in terms of the implied constants.

% \begin{lem}[2.5 of \cite{SalbCrelle}, explicit for planar curves]
% In Lemma 2.5 of \cite{SalbCrelle} for planar curves, assume that $X_p$ is irreducible. Then we can take $O_{d,n}(s)= s\mu/2 + \mu^3$ (more or less).
% \end{lem}
% \begin{proof}
% Let $R_0$ be the ring $\ZZ_p[x_0,x_1,x_2]/(f)$, and let $R$ be the $p$-adic completion

% \end{proof}
%
%\begin{lem}[Thm 2.2 of \cite{Walsh}, explicit for planar curves]
%
%\end{lem}
%\begin{proof}
%
%
%If $X_p$ is irreducible of degree $d$, use Katz' bounds of Theorem 12 of \cite{Katz-Bet} on Betti numbers to get a nice Lang-Weil. Even better: there are at most $d(d-1)$ singular points on $X$.
%
%
%\end{proof}

\section{Points on projective hypersurfaces \`a la Walsh}\label{sec3}\label{sec:Walsh}

%By Lemma \ref{lem:not-geom-int} it is
%enough to show Theorem \ref{thm:W1.1} in the case where $X$ is geometrically integral.

%\raf{I have the overall impression that only the surface (and curve) case in $\PP^3$ is needed, as we work towards adapting 3.16 of \cite{Salberger-dgc}. That is, I think we can restrict to $n\leq 3$.}

\subsection{Formulation of main result}

The following result is the goal of this section and an improvement to Theorem 1.3 of \cite{Walsh}.
Call a polynomial $f$ over $\ZZ$ primitive if the greatest common divisor of its coefficients equals $1$.
For any $f$, we write $\lVert f\rVert$ for the maximum of the absolute values of the coefficients of $f$.

%%raf: do we need d>1 here?
%\begin{thm}\label{thm:walsh-hypers}
%Let $n>0$ be an integer. Then there exists $c$ and $e$ such that the following holds for all choices of $f,d,B$.
%Let $f$ be any irreducible homogeneous polynomial in $\ZZ[x_0,\ldots,x_{n+1}]$ of degree $d\geq 1$. Suppose that $f$ is primitive. Write $X$ for the hypersurface in $\PP_\QQ^n$ cut out by $f$. Choose $B\geq 1$.  Then there exists a homogeneous $g$ in $\ZZ[x_0,\ldots,x_{n+1}]$ of degree
%at most
%$$
%cd^e \big( B^{\frac{n+1}{nd^{1/n}}} ( \log \lVert f\rVert  + 1 ) \lVert f\rVert^{ -1/nd^{1+1/n}}  +   \log B \big),
%$$
%not divisible by $f$, and vanishing at all points on $X$ of height at most $B$.
%\end{thm}
\begin{thm}\label{thm:walsh-hypers-precise}
Let $n>0$ be an integer. Then there exists $c$ such that the following holds for all choices of $f,d,B$.
Let $f$ be a primitive irreducible homogeneous polynomial in $\ZZ[x_0,\ldots,x_{n+1}]$ of degree $d\geq 1$, and write $X$ for the hypersurface in $\PP_\QQ^{n+1}$ cut out by $f$. Choose $B\geq 1$.  Then there exists a homogeneous $g$ in $\ZZ[x_0,\ldots,x_{n+1}]$ of degree
at most
\[
c B^{\frac{n+1}{nd^{1/n}}} \frac{d^{4-1/n} b(f)}{\lVert f \rVert^{\frac{1}{n} \frac{1}{d^{1+1/n}}}}  +   c d^{1-1/n} \log B + cd^{4-1/n},
\]
not divisible by $f$, and vanishing at all points on $X$ of height at most $B$.
\end{thm}

Here the quantity $b(f)$ will be defined in Definition \ref{defn:badness}; it always satisfies $b(f) \leq O(\max(d^{-2} \log\lVert f \rVert, 1))$.
The main improvement over \cite{Walsh} lies in the polynomial dependence on the degree $d$.
%We give even more detailed versions of Theorem \ref{thm:walsh-hypers} below, see Theorem \ref{thm:walsh-hypers-precise}, in which it is also possible to remove the term $\log\lVert f\rVert+1$ in favour of a term involving the primes $p$ modulo which $f$ fails to be absolutely irreducible.

We also immediately obtain the following, which is the essential tool for proving Theorem \ref{thm:W1.1}.
\begin{cor}\label{cor:precise_Walsh_planar_curve}
  For any primitive irreducible polynomial $f \in \ZZ[x_0, x_1, x_2]$ homogeneous of degree $d$ and any $B \geq 1$ we have
  \[ N(f, B) \leq c B^{\frac{2}{d}} \frac{d^4 b(f)}{\lVert f \rVert^{1/d^2}} + c d \log B + c d^4 \leq c' d^4 B^{\frac{2}{d}} ,\]
  where $c$, $c'$ are absolute constants.
\end{cor}
\begin{proof}
  Apply Theorem \ref{thm:walsh-hypers-precise} to obtain a polynomial $g$, and then apply B\'ezout's theorem to the curves defined by $f$ and $g$.
  This yields the first inequality.
  For the second inequality we can use that $b(f)/\lVert f \rVert^{1/d^2}$ is bounded because $b(f) \leq O(\max(d^{-2} \log\lVert f \rVert, 1))$.
\end{proof}

\subsection{A determinant estimate}

In this section we want to use the results of Section \ref{sec:salberger} for a number of primes simultaneously.
It is useful to introduce the following measure of the set of primes modulo which an absolutely irreducible polynomial over the integers ceases to be absolutely irreducible.
\begin{definition}\label{defn:badness}
  For an integer polynomial $f$ in an arbitrary number of variables we set $b(f)=0$ if $f$ is not absolutely irreducible, and
  \[ b(f) = \prod_p \exp(\frac{\log p}{p}) \]
  otherwise, where the product is over those primes $p > 27d^4$ such that the reduction of $f$ modulo $p$ is not absolutely irreducible.
\end{definition}

For now we work with a degree $d$ hypersurface in $\PP^{n+1}$ defined by a primitive polynomial $f \in \ZZ[x_0, \dotsc, x_{n+1}]$ which is absolutely irreducible.
We first establish a basic estimate on $b(f)$, showing in particular that it is finite.

%\begin{thm}[Explicit Noether Polynomials]
%  There exists a collection of polynomials $\Phi_t$ for testing absolute irreducibility, with each $\Phi_t$ having degree at most $12d^6$ and sum of absolute values of coefficients at most $(2d)^{12d^7+12d^6 n + 32d^6} \leq (2d)^{20n d^7}$.
%\end{thm}

%\begin{lem}\label{lem:primes_with_bounded_product}
%  Let $P$ be a set of prime numbers with product not exceeding $C$. Then $\sum_{p \in P} \log p/p \leq \log\log C + 3$
%\end{lem}
%\begin{proof}
%  Let $P'$ be the subset of $P$ consisting of those elements larger than $\log C$.
%  Then \[ \sum_{p \in P} \frac{\log p}{p} \leq \sum_{p \leq \log C} \frac{\log p}{p} + \sum_{p \in P'} \frac{\log p}{p} \leq \log\log C + 2 + \sum_{p \in P'} \frac{\log p}{\log C} \leq \log\log C + 3, \] where we have used Mertens' first theorem in the second estimate.
%\end{proof}
\begin{thm}[Explicit Noether polynomials, {\cite[Satz 4]{RuppertCrelle}}] \label{thm:explicit_noether}
  Let $d \geq 2$, $n \geq 3$. There is a collection of homogeneous polynomials $\Phi$ in $\binom{n+d}{n}$ variables over $\ZZ$ of degree $d^2-1$, such that
  \[ \lVert\Phi\rVert_1 \leq d^{3d^2-3} \big[ \binom{n+d}{n} 3^d\big]^{d^2-1} \]
  (where $\lVert \cdot \rVert_1$ denotes the sum of the absolute values of the coefficients), and such that the following holds for any polynomial $F$ in $n+1$ variables homogeneous of degree $d$ over any field:
  \begin{itemize}
  \item if $F$ is not absolutely irreducible, then all $\Phi$'s vanish when applied to the coefficients of $F$, reducing modulo the characteristic of the ground field if necessary; 
  \item if $F$ is absolutely irreducible over a field of characteristic $0$, then one of the $\Phi$'s does not vanish when applied to the coefficients of $F$.
  \end{itemize}
\end{thm}
\begin{cor}\label{cor:estimate_bad_primes}
 $ b(f) %\leq d^{3d^2-3} \big( \binom{n+d}{n} 3^d\big)^{d^2-1} \lVert f\rVert^{d^2-1}
     \leq O(\max(d^{-2} \log\lVert f \rVert, 1))$.
\end{cor}
\begin{proof}
  Write $\cP$ for the set of prime numbers $p > 27d^4$ modulo which $f$ is not absolutely irreducible.
  There exists a Noether form $\Phi$ with coefficients in $\ZZ$ such that $\Phi$ applied to the coefficients of $f$ is non-zero, but is divisible by any prime in $\cP$. In particular, the product of such $p$ is bounded by $c := \lVert\Phi\rVert_1 \lVert f \rVert^{\deg\Phi}$.
  Now \begin{align*}
  \log b(f)
  &= \sum_{p \in \cP} \frac{\log p}{p} \leq \sum_{27 d^4 < p \leq \log c} \frac{\log p}{p} + \sum_{\log c < p \in \cP} \frac{\log p}{\log c} \\
  &\leq \max(\log \log c - 4 \log d, 0) + O(1) + \frac{\log c}{\log c} \\
  &\leq \max(\log \log c - 4\log d, 0) + O(1) \\
  &\leq \max(\log(\deg\Phi \log\lVert f\rVert) - 4\log d, \log\log\lVert \Phi\rVert_1 - 4\log d, 0) + O(1)
  ,\end{align*}
  where we have used that the function $\log x - \sum_{p \leq x} \frac{\log p}{p}$ is bounded (Mertens' first theorem).
  Since $\log\log\lVert\Phi\rVert_1 - 4\log d$ is bounded above, the claim follows.
\end{proof}

We now adapt \cite[Theorem 2.3]{Walsh}, keeping track of the dependency on the degree and on $b(f)$.

\begin{lem}\label{lem:chebyshev}
  For any $x>0$, $\sum_{p \leq x} \log p \leq 2 x$, where the sum extends over prime numbers not exceeding $x$.
\end{lem}
\begin{proof}
  This is a classical estimate on the first Chebyshev function.
\end{proof}
\begin{lem}
  As $x$ varies over positive real numbers we have $\sum_{p > x} \frac{\log p}{p^{3/2}} = O(x^{-1/2})$, where the sum extends over prime numbers greater than $x$.
\end{lem}
\begin{proof}
  Estimate the density of prime numbers using the prime number theorem and compare the sum with an integral.
\end{proof}

\begin{prop}\label{prop:walsh_sec_2}
  Let $(\xi_1, \dotsc, \xi_s)$ be a tuple of rational points in $X$, let $F_{li} \in \ZZ[x_0, \dotsc, x_{n+1}]$, $1 \leq l \leq L$, $1 \leq i \leq s$, be homogeneous polynomials with integer coefficients, and write $\Delta_l$ for the determinant of $(F_{li}(\xi_j))_{ij}$. Let $\Delta$ be the greatest common divisor of the $\Delta_l$, and assume that $\Delta \neq 0$.
  Then we have the bound
  \[ \log \lvert\Delta\rvert \geq \frac{n!^{1/n}}{n+1} s^{1+1/n}(\log s - O_n(1) - n (4\log d + \log b(f)) )  .\]
\end{prop}
This is a more explicit variant of \cite[Theorem 2.3]{Walsh}.
\begin{proof}
  Let $\cP$ be the collection of prime numbers $p$ such that either $p \leq 27d^4$ or $X_p$ is not geometrically integral.
%  By Lemma \ref{lem:chebyshev}, we obtain the estimate
 % \[ \prod_{p \in \cP} p \leq \exp(54d^4) b(f). \]

  We now apply Corollary \ref{cor:Salberger_determinant} to all prime numbers $p \leq s^{1/n}$ not in $\cP$, yielding
  \[ \log\lvert\Delta\rvert \geq \frac{n!^{1/n} n}{n+1} s^{1+1/n} \sum_{\cP \not\ni p \leq s^{1/n}} \frac{\log p}{p + O_{n}(d^2p^{1/2})} - O_n(s) \sum_{p \leq s^{1/n}} \log p .\]
  The last term is bounded by $O_n(1) s^{1+1/n}$.

  In estimating the main term, we may use that $\frac{1}{p + O_n(d^2 p^{1/2})} \geq \frac{1}{p} - O_n(d^2) \frac{1}{p^{3/2}}$. We can then bound
  \begin{align*}
    &\sum_{\cP \not\ni p \leq s^{1/n}} \frac{\log p}{p + O_{n}(d^2 p^{1/2})} \\
    &\geq \sum_{p \leq s^{1/n}} \frac{\log p}{p} - \sum_{p \in \cP} \frac{\log p}{p} - O_{n}(d^2) \sum_{\cP \not\ni p \leq s^{1/n}} \frac{\log p}{p^{3/2}} \\
    &\geq \frac{\log s}{n} - \sum_{p \leq 27d^4} \frac{\log p}{p} - \log b(f) - O(1) - O_{n}(d^2\sum_{p > 27d^4} \frac{\log p}{p^{3/2}}) \\
    &\geq \frac{\log s}{n} - \log (27d^4) - \log b(f) - O(1) - O_{n}(d^2 (27d^4)^{-1/2}) \\
    &\geq \frac{\log s}{n} - 4\log d - \log b(f) - O_n(1). \qedhere
  \end{align*}
\end{proof}

\subsection{The main estimates}

%Goal:
%Make Walsh's Sections 4 explicit and polynomial in $d$, using our previous work.

%The following is a more precise version of Theorem \ref{thm:walsh-hypers}. Note that we are using the letter $B$ for the height bound throughout where Walsh uses $N$.

We first establish that we can reduce to the case of absolutely irreducible $f$ in the proof of Theorem \ref{thm:walsh-hypers-precise}.
\begin{lem}\label{lem:walsh_wlog_abs_irred}
  If $f \in \ZZ[x_0, \dotsc, x_{n+1}]$ is homogeneous of degree $d \geq 1$ and irreducible but not absolutely irreducible, then there exists another polynomial $g \in \ZZ[x_0, \dotsc, x_{n+1}]$ of degree $d$, not divisible by $f$, which vanishes on all rational zeroes of $f$.
\end{lem}
\begin{proof}
  This is established in the first paragraph of Section 4 of \cite{Walsh}.
\end{proof}

Let us now work with a restricted class of homogeneous polynomials $f$, namely those which are absolutely irreducible and for which the leading coefficient $c_f$, i.e.~the coefficient of the monomial $x_{n+1}^d$, satisfies \[c_f \geq \lVert f \rVert C^{-n d^{1+1/n}}\] for some positive constant $C$ which is allowed to depend on $n$ (for this reason the factor $n$ in the exponent is in fact superfluous, but it simplifies the proof write-up below).

The two main results are the following.
\begin{lem}\label{lem:Walsh_small_f}
  For $f$ as above, and $B$ satisfying $\lVert f \rVert \leq B^{2d(n+1)}$, there exists a homogeneous polynomial $g$ not divisible by $f$, vanishing at all zeroes of $f$ of height at most $B$, and of degree
  \[ M = O_n(1) B^{\frac{n+1}{nd^{1/n}}} \frac{d^{4-1/n} b(f)}{\lVert f \rVert^{n^{-1}d^{-1-1/n}}} + d^{1-1/n} \log B + O_n(d^2) .\]
\end{lem}

\begin{lem}\label{lem:Walsh_big_f}
  For $f$ as above, and $B$ satisfying $\lVert f \rVert \geq B^{2d(n+1)}$, there exists a homogeneous polynomial $g$ not divisible by $f$, vanishing at all zeroes of $f$ of height at most $B$, and of degree
  \[ M = O_n(d^{4-1/n}) .\]
\end{lem}
These two lemmas together clearly imply the statement of Theorem \ref{thm:walsh-hypers-precise}, at least for polynomials $f$ satisfying the condition on leading coefficients.

We follow the exposition in \cite[Section 4]{Walsh}, and prove the two lemmas together. We shall need the following.
\begin{thm}[{\cite[Theorem 1]{Bombi-Vaal}}]\label{thm:bombval_siegel}
  Let $\sum_{k=1}^r a_{mk}x_k=0$ ($m=1, \dotsc, s$) be a system of $s$ linearly independent equations in $r>s$ variables $x_1, \dotsc, x_r$, with coefficients $a_{mk} \in \ZZ$.
  Then there exists a non-trivial integer solution $(x_1, \dotsc, x_r)$ satisfying
  \[ \max_{1\leq i \leq r} \lvert x_i\rvert \leq (D^{-1} \sqrt{\lvert \det(AA^\top)\rvert})^{\frac{1}{r-s}} .\]
  Here $A=(a_{mk})$ is the matrix of coefficients and $D$ is the greatest common divisor of the determinants of the $s \times s$ minors of $A$.
\end{thm}

\begin{proof}[Proof of Lemmas \ref{lem:Walsh_small_f} and \ref{lem:Walsh_big_f}]
Fix $B\geq 1$, and let $S$ be the set of rational points on the hypersurface described by $f$ of height at most $B$.
Let $M>0$ be such that there is no homogeneous polynomial $g$ of degree $M$, not divisible by $f$, which vanishes on all points in $S$; we shall show that $M$ is bounded in terms of $n, B, d, \lVert f \rVert$ as stated.
Let us assume in the following that $M$ is bigger than some constant (to be specified later) times $d^2$.

Given an integer $D$, write $\cB[D]$ for the set of monomials of degree $D$ in $n+2$ variables, so $\lvert\cB[D]\rvert = \binom{D+n+1}{n+1}$.
Write $\Xi \subseteq S$ for a maximal subset which is algebraically independent over monomials of degree $M$, in the sense that applying all monomials in $\cB[M]$ to $\Xi$ yields $s = \lvert\Xi\rvert$ linearly independent vectors.
Let $A$ be the $s \times r$ matrix whose rows are these vectors, where $r = \lvert\cB[M]\rvert = \binom{M+n+1}{n+1}$; each entry of $A$ is bounded in absolute value by $B^M$.
Since all polynomials in $f \cdot \cB[M-d]$ vanish on $\Xi$, and no polynomials of degree $M$ not divisible by $f$ do by assumption on $M$, we have $s = \lvert\cB[M]\rvert - \lvert\cB[M-d]\rvert$.

Now $A$ describes a system of linear equations whose solutions correspond to (the coefficients of) homogeneous polynomials of degree $M$ vanishing on all points in $\Xi$ and therefore all points in $S$; by assumption, these polynomials are multiples of $f$ and therefore have one coefficient of size at least $c_f \geq \lVert f \rVert C^{-n d^{1+1/n}}$ by the assumption on $f$.
Hence Theorem \ref{thm:bombval_siegel} yields
\[  \Delta \leq \sqrt{\lvert\det(AA^\top)\rvert} (\lVert f \rVert C^{-n d^{1+1/n}})^{s-r}  ,\]
where we write $\Delta$ for the greatest common divisor of the determinants of the $s \times s$ minors of $A$.
Taking logarithms, using the estimate $\lvert\det(AA^\top)\rvert \leq s! (r B^M)^s$ obtained by estimating the size of the coefficients of $AA^\top$, and using the estimate for $\Delta$ obtained from Proposition \ref{prop:walsh_sec_2}, this expands as follows:
\begin{multline*}  \frac{n!^{1/n}}{n+1} s^{1+1/n}(\log s - O_n(1) - n (4\log{d} + \log b(f)) ) \\ \leq \frac{\log s!}{2} + \frac{s}{2} \log r + sM \log B  - (r-s)\big( \log\lVert f\rVert - n d^{1+1/n}O_n(1) \big) \end{multline*}

We can use the estimates $\log s! \leq s \log s$ and $\log r \leq \log (M+1)^{n+1} \leq O_n(\log M) \leq O_n(\log s)$ to see that the first two terms on the right-hand side are both in $O_n(s^{1+1/n})$ and can hence be neglected by adjusting the constant $O_n(1)$ on the left-hand side.
Diving by $Ms$ now yields:
\begin{multline}\label{ineqn:walsh_b} \frac{n!^{1/n}}{n+1} \frac{s^{1/n}}{M}(\log s - O_n(1) - n (4\log{d} + \log b(f))) \\ \leq \log B - \frac{r-s}{Ms} \big( \log\lVert f\rVert - n d^{1+1/n}O_n(1) \big) \end{multline}

The term $s = \binom{M+n+1}{n+1} - \binom{M-d+n+1}{n+1}$ is a polynomial in $M$ and $d$. We can write \[ s = \frac{dM^n}{n!} + O_n(d^2 M^{n-1}) ,\]
in particular $\log s = \log d + n \log M - O_n(1)$.
By rearranging and applying the binomial series, which is legal since $d^2/M$ is bounded above by an adjustable absolute constant, we also obtain
\[ \frac{s^{1/n}}{M} = \frac{d^{1/n}}{n!^{1/n}} + O_n(\frac{d^2}{M}). \]

Thus the left-hand side of the inequality above can be replaced by
\[ \frac{d^{1/n}n}{n+1}(\log M - O_n(1) - ((4-1/n)\log{d} + (1+O_n(\frac{d^{2-1/n}}{M}))\log b(f))),\]
where we have dropped terms $O_n(d^{2-1/n}\log M /M)$ and $O_n(d^{2-1/n} \log d / M)$ by adjusting the constant in $O_n(1)$.

Let us now estimate $\frac{r-s}{Ms}$. We have $r-s = \frac{M^{n+1}}{(n+1)!} + O_n(dM^n)$, so \[ \frac{r-s}{Ms} = \frac{1}{d(n+1)} \frac{1 + O_n(d/M)}{1+O_n(d/M)} = \frac{1}{d(n+1)} + O_n(\frac{1}{M}) .\]

Therefore inequality \eqref{ineqn:walsh_b} becomes 
\begin{multline}\label{ineqn:walsh_c}
  \frac{d^{1/n}n}{n+1}(\log M - O_n(1) - ((4-1/n)\log{d} + (1+O_n(\frac{d^{2-1/n}}{M}))\log b(f)))
  \\ \leq \log B - \frac{\log\lVert f\rVert}{d(n+1)} - O_n(\frac{\log\lVert f\rVert}{M}).
\end{multline}

Let us now assume that $\lVert f \rVert \leq B^{2d(n+1)}$ and $M \geq d^{1-1/n} \log B$.
Then $\log \lVert f \rVert \leq 2d(n+1) \log B \leq O_n(d^{1/n} M)$, so we can drop the last term on the right-hand side, as well as the $O_n(\frac{\log b(f)}{M})$ on the left-hand side.
Rearranging yields that \[ \log M \leq O_n(1) + \frac{n+1}{d^{1/n}n} \log B - \frac{\log \lVert f\rVert}{n d^{1+1/n}} + (4-1/n)\log d + \log b(f) ,\]
so we obtain Lemma \ref{lem:Walsh_small_f}.

Now, on the other hand, assume that $\lVert f\rVert \geq B^{2d(n+1)}$ and $M \geq 4d(n+1)$.
Rearranging inequality \eqref{ineqn:walsh_c} yields
\begin{align*}
\log M &\leq O_n(1) + (4-1/n)\log{d} + (1+O_n(d^{-1/n})) \log b(f) - \frac{\log \lVert f \rVert}{4 n d^{1+1/n}} \\
&\leq O_n(1) + \max\{3\log{d}, (4-1/n) \log{d} \},
\end{align*}
where we have used
\begin{align*}
  & O_n(1) \log b(f) - \frac{\log \lVert f \rVert}{4 n d^{1+1/n}} \\
  &\leq O_n(1) \max(\log\log\lVert f\rVert - 2\log d, 0) - \frac{\log \lVert f \rVert}{4 n d^{1+1/n}} \\
  &\leq \max(0, -2\log d + \log{O_n(1) 4n d^{1+1/n}}) \\
  &\leq O_n(1)
\end{align*}
by Corollary \ref{cor:estimate_bad_primes} and the lemma below.
%where we have used Corollary \ref{cor:estimate_bad_primes} and
%\begin{align*}
%&(1+O_n(d^{-1/n}))(\log\log \lVert f \rVert + O(1) - 2\log d) - \frac{\log \lVert f \rVert}{4d(n+1)} \\
%&\leq (1+O_n(d^{-1/n}))(O_n(1) + \log[(1+O_n(d^{-1/n}))4(n+1)d^{1+1/n}]) - 2\log d - O(1) \\
%&\leq O_n(1) + (-1+1/n)\log d ,\end{align*}
%using the lemma below.
This establishes Lemma \ref{lem:Walsh_big_f}.
\end{proof}

\begin{lem}
  Let $c > 0$. For any $x>1$ we have $\log\log x - \log x/c \leq \log c + O(1)$.
\end{lem}
\begin{proof}
  Let $C = \sup_{x>1} (\log \log x - \log x)$; note that the supremum exists, since it is taken over a continuous function on $\left]1, \infty\right[$ which tends to $-\infty$ at both ends of the interval.
  Now $\log \log x - \log x/c = \log c + \log \log x^{1/c} - \log x^{1/c} \leq \log c + C$.
\end{proof}

\subsection{Finishing the proof}

We use the ideas from \cite[Section 3]{Walsh} to finish the proof of Theorem \ref{thm:walsh-hypers-precise}.

\begin{lem}
  Let $f \in \CC[x]$ be a polynomial of degree $\leq d$, and write $\lVert f\rVert$ for the maximal absolute value among the coefficients.
  There exists an integer $a$, $0 \leq a \leq d$, such that $\lvert f(a) \rvert \geq 3^{-d} \lVert f\rVert$.
\end{lem}
\begin{proof}
  This is a statement about the $\lVert\cdot\rVert_\infty$-operator norm of the inverse of the Vandermonde matrix with nodes $0, \dotsc, d$, which can be deduced from \cite[Theorem 1]{Gautschi}.
\end{proof}

\begin{lem}
  Let $f \in \CC[x_0, \dotsc, x_{n+1}]$ be homogeneous of degree $d$. There exist integers $a_0, \dotsc, a_n$ with $0 \leq a_i \leq d$ such that $\rvert f(a_0, \dotsc, a_n, 1) \lvert \geq 3^{-(n+1)d} \lVert f \rVert$.
\end{lem}
\begin{proof}
  Dehomogenize by setting $x_{n+1} = 1$, and then use induction with the preceding lemma.
\end{proof}

\begin{proof}[Proof of Theorem \ref{thm:walsh-hypers-precise}]
  Take a non-zero $f \in \ZZ[x_0, \dotsc, x_{n+1}]$ homogeneous of degree $d$. Consider $a_0, \dotsc, a_n$ as in the last lemma and let $A = I + A_0 \in \operatorname{SL}_{n+2}(\ZZ)$, where $I$ is the $(n+2) \times (n+2)$ identity matrix and $A_0$ has its last column equal to $(a_0, \dotsc, a_n, 0)$ and zero everywhere else. Note that $A^{-1} = I - A_0$.
%The inverse $A^{-1}$ has integer coefficients bounded in absolute value by $(n+1)! \cdot d^{n+1}$.

  Let $f'= f \circ A$. By construction, the $x_{n+1}^d$-coefficient of $f'$ is $\geq 3^{-(n+1)d} \lVert f \rVert$.
Because of the boundedness of the entries of $A$, we furthermore see that 
\[ \lVert f'\rVert \leq { n + d + 1 \choose n + 1} d^{n+1} \lVert f \rVert. \]
In particular, the $x_{n+1}^d$-coefficient of $f'$ is greater than $C^{-n d^{1+1/n}} \lVert f' \rVert$ for some constant $C$ depending only on $n$.
  The polynomial $f'$ is primitive if and only if $f$ is, since they are related by the matrices $A$, $A^{-1}$ with integer coefficients, and $b(f)=b(f')$.
Furthermore, if $g'$ is a homogeneous polynomial in $\ZZ[x_0, \dotsc, x_{n+1}]$ vanishing on all zeroes of $f'$ up to a certain height $B'$, then $g = g' \circ A^{-1}$ is a polynomial of the same degree vanishing on all zeroes of $f$ up to height $B = B' / (d+1)$.

  Since either Lemma \ref{lem:Walsh_small_f} or Lemma \ref{lem:Walsh_big_f} applies to $f'$ and $B'$, we obtain the desired statement for $f$.
%  Let $n, d, B, f$ be given. We construct $f'$ as above, obtaining \[ 3^{-(n+1)d} \lVert f \rVert \leq \lVert f' \rVert \leq (n+2) d^{n+2} \lVert f \rVert ,\]
%  and let $B' = c_n d^{n+1} B$ with a large constant $c_n$ as above.
%  Either Lemma \ref{lem:Walsh_small_f} or Lemma \ref{lem:Walsh_big_f} applies to $f'$ and $B'$, and by substituting back we obtain the claim for $f$ and $B$.
\end{proof}
%Take non-zero $f \in \ZZ[X_0, \dotsc, X_{n+1}]$ homogeneous of degree $d$. Consider $x_0, \dotsc, x_n$ as in the last lemma and let $A = I + B \in \operatorname{SL}_{n+2}(\ZZ)$, where $I$ is the $(n+2) \times (n+2)$ identity matrix and $B$ has its last column equal to $(x_0, \dotsc, x_n, 0)$ and zero everywhere else.
%The inverse $A^{-1}$ has integer coefficients bounded in absolute value by $(n+1)! d^{n+1}$.
%
%Let $f'= f \circ A$.
%By construction, the $x_{n+1}^d$-coefficient of $f'$ is $\geq 3^{-(n+1)d} \lVert f \rVert$.
%Because of the boundedness of the entries of $A$, we furthermore see that $\lVert f'\rVert \leq (n+2) d^{n+2} \lVert f \rVert$.
%In particular, the $x_{n+1}^d$-coefficient of $f'$ is greater than $C_n^{d^{1+1/n}} \lVert f' \rVert$ for a sufficiently small constant $C_n$ depending only on $n$.
%
%The polynomial $f'$ is primitive if and only if $f$ is, since they are related by the matrices $A$, $A^{-1}$ with integer coefficients, and furthermore $b(f)=b(f')$.
%If $g'$ is a homogeneous polynomial in $\ZZ[X_0, \dotsc, X_{n+1}]$ vanishing on all zeroes of $f'$ up to a certain height $B'$, then $g = g' \circ A^{-1}$ is a polynomial of the same degree vanishing on all zeroes of $f$ up to height $B = B' c_n d^{-n-1}$ for some constant $c_n$ depending only on $n$.
%

%\input{sec4.tex}

\section{Proofs of Theorems \ref{thm:dcgdegree}, \ref{thm:W1.1}, \ref{thm:Bombieri-Pila}, \ref{thm:0.4}, \ref{thm:bhargava_and_co}}\label{sec5}

%(See list of implications from email.)

\subsection{On trivial bounds}\label{sec:triv}

In this subsection, we extend our notation to varieties defined over any field $K$ containing $\QQ$, and we write $N(X,B)$ for the number of points in $\PP^n(\QQ)\cap X(K)$  of height at most $B$, when $X$ is a subvariety of $\PP^n_K$, and similarly we write $N_{\rm aff}(Y,B)$ for the number of points in $\ZZ^n\cap Y(K)\cap [-B,B]^n$, when $Y\subseteq \AA^n_K$.

\begin{lem}\label{lem:gen_Schwarz_Zippel}
  Let $X \subseteq \AA^n_{\overline{\QQ}}$ be a (possibly reducible) variety of pure dimension $m$ and degree $d$ defined over $\overline{\QQ}$. Then the number $N_{\rm aff}(X,B)$ of integral points on $X$ of height at most $B$ is bounded by $d(2B+1)^m$.
\end{lem}
When $X$ is a hypersurface, this is the well-known Schwarz-Zippel bound, and even the general case appears in many places in the literature, albeit often without making the bound completely explicit.
\begin{proof}
  This is an easy inductive argument using intersections with shifts of coordinate hyperplanes.
  In fact, the proof of \cite[Theorem 1]{BrowningHeathBrown-Crelle} automatically gives this stronger statement.
\end{proof}

\begin{cor}\label{cor:find_point}
  For an irreducible affine variety $X$ in $\AA^n$ of degree $d$ and dimension $< n$ there exists a tuple $(a_1, \dotsc, a_n)$ of integers not on $X$, with $\lvert a_i \rvert \leq d$ for every $i$.
  For every irreducible projective variety $X$ in $\PP^n$ of degree $d$ and dimension $< n$ there exists a point in $\PP^n(\QQ)$ of height at most $d$ not on $X$.
\end{cor}
\begin{proof}
  The affine version is implied by the preceding lemma, and the projective version follows by considering the affine cone.
\end{proof}

\begin{lem}\label{lem:Gal}
 Let $X \subseteq \AA^n_{\overline{\QQ}}$ be an absolutely irreducible variety of dimension $m$ and degree $d$ not defined over $\QQ$. Then the number  $N_{\rm aff}(X,B)$  of integral points on $X$ of height at most $B$ is bounded by $d^2(2B+1)^{m-1}$.
\end{lem}
By considering the affine cone over a projective variety, this result also applies to projective varieties of dimension $m$, with bound $d^2(2B+1)^m$.
\begin{proof}
  For every field automorphism $\sigma$ of $\overline{\QQ}$, there is a conjugate variety $X^\sigma$. Since $X$ is not defined over $\QQ$, there exists a $\sigma$ with $X^\sigma \neq X$.
  All $\QQ$-points of $X$ necessarily also lie on $X^\sigma$.  Since $X^\sigma$ has degree $d$, it is the intersection of hypersurfaces of degree $\leq d$, see for instance \cite[Proposition 3]{Heintz_DefinabilityAndFastQE}. Let $Y$ be a hypersurface of degree $\leq d$ containing $X^\sigma$ and not containing $X$. Then $X \cap Y$ is a variety of pure dimension $m-1$ and degree at most $d^2$.
  Now Lemma \ref{lem:gen_Schwarz_Zippel} gives the result.
\end{proof}

The following allows us to reduce to the geometrically irreducible situation when counting points on varieties.
\begin{cor}\label{cor:reduction_geometrically_irreducible}
  Let $X \subseteq \AA^n$ be an irreducible variety over $\QQ$ of dimension $m$ and degree $d$ which is not geometrically irreducible.
  Then for any $B \geq 1$ we have $N_{\rm aff}(X,B) \leq d^2 (2B+1)^{m-1}$.
\end{cor}
As above, this also applies to projective varieties.
\begin{proof}
  Let $K/\QQ$ be a finite Galois extension over which $X$ splits into absolutely irreducible components, and let $Y$ be one of the components. Since all components are Galois-conjugate, the $\QQ$-points on $X$ in fact also lie on $Y$.
  Now the preceding lemma applied to $Y$ gives the result.
\end{proof}

\begin{remark}\label{rem:wlog_abs_irred}
  Note that this trivially proves Theorems \ref{thm:dgcdegree} and \ref{thm:Bombieri-Pila} for irreducible, but not geometrically irreducible varieties, and similarly for absolutely irreducible varieties defined over $\overline{\QQ}$ but not over $\QQ$.
  The same applies for Theorem \ref{thm:W1.1} by considering a projective curve as the union of an affine curve with a finite number of points.

  Thus we henceforth only need to concern ourselves with absolutely irreducible varieties defined over $\QQ$.
\end{remark}

\subsection{Affine counting} \label{subsec:affinecounting}

Our results for projective hypersurfaces from the last section yield the following result for affine hypersurfaces, by refining the technique given in \cite[Remark 2.3]{Ellenb-Venkatesh}.
\begin{prop}\label{prop:Ellenb}
Fix an integer $n>0$. Then there exist $c$ and $e$ such that the following holds for all $f,B,d$.
Let $f$ be in $\ZZ[x_1,\ldots,x_{n+1}]$ be irreducible, primitive and of degree $d$.  For each $i$ write $f_i$ for the degree $i$ homogeneous part of $f$. Fix $B\geq 1$. Then there is a polynomial $g$ in $\ZZ[x_1,\ldots,x_{n+1}]$  of degree at most
\[
c B^{\frac{1}{d^{1/n}}} d^{2-1/n} \frac{\min(\log \lVert f_d \rVert +d\log B + d^2, d^2 b(f))}{ \lVert f_d \rVert^{\frac{1}{n} \frac{1}{d^{1+1/n}}}}  +   c d^{1-1/n} \log B + cd^{4-1/n},
\]
not divisible by $f$, and vanishing on all points  $x$ in $\ZZ^{n+1}$ satisfying $f(x)=0$ and $|x_i|\leq B$.
\end{prop}
To prove Proposition \ref{prop:Ellenb} we need the following lemmas.
 \begin{lem}[{\cite[Lemma 5]{Brow-Heath-Salb}}] \label{lem:bounded_coefficients_or_vanishing_g}
   Let $f \in \ZZ[x_1, \dotsc, x_{n+2}]$ be a primitive absolutely irreducible polynomial, homogeneous of degree $d$, defining a hypersurface $Z$ in $\PP^{n+1}$. Let $B \geq 1$. Then either the height of the coefficients of $f$ is bounded by $O_n(B^{d \binom{d+n+1}{n+1}})$, or there exists a homogeneous polynomial $g$ of degree $d$ vanishing on all points of $Z$ of height at most $B$.
 \end{lem}

 \begin{lem}\label{lem:ellenberg_estimate}
   For $F \in \ZZ[x_1, \dotsc, x_{n+2}]$ an irreducible primitive homogeneous polynomial and $1 \leq y \leq \lVert F\rVert$ we have
   \[ d^{4-1/n} \frac{b(F)}{\lVert F\rVert^{\frac{1}{n}\frac{1}{d^{1+1/n}}}} \leq O_n(1) d^{2-1/n} \frac{\log y + d^2}{y^{\frac{1}{n}\frac{1}{d^{1+1/n}}}} .\]
 \end{lem}
 \begin{proof}
   The function
   $$
   x \mapsto \frac{\log x}{x^{\frac{1}{n}\frac{1}{d^{1+1/n}}}}
   $$
   on $(1, \infty)$ is monotonically increasing up to its maximum when $x^{\frac{1}{n}\frac{1}{d^{1+1/n}}} = e$, and monotonically decreasing thereafter.

   Let us write $x = \lVert F\rVert$ and use $d^2 b(F) \leq O_n(1) (\log x + d^2)$ by Corollary \ref{cor:estimate_bad_primes}.
   By the monotonicity considered above, there is nothing to show when $y^{\frac{1}{n}\frac{1}{d^{1+1/n}}} \geq e$.
   Otherwise, \[ 2 d^{2-1/n} \frac{\log y + d^2}{y^{\frac{1}{n}\frac{1}{d^{1+1/n}}}} \geq 2 d^{2-1/n} \frac{d^2}{y^{\frac{1}{n}\frac{1}{d^{1+1/n}}}} \geq d^{4-1/n}(\frac{1}{e} + \frac{1}{y^{\frac{1}{n}\frac{1}{d^{1+1/n}}}}), \]
   and the left-hand side of the inequality in the statement is always bounded by \[ O_n(1) d^{2-1/n} \frac{\log x + d^2}{x^{\frac{1}{n}\frac{1}{d^{1+1/n}}}} \leq O_n(1) d^{2-1/n} (\frac{n d^{1+1/n}}{e} + \frac{d^2}{x^{\frac{1}{n}\frac{1}{d^{1+1/n}}}}),  \]
   yielding the claim.
\end{proof}

\begin{proof}[Proof of Proposition \ref{prop:Ellenb}]
By applying Lemma \ref{lem:walsh_wlog_abs_irred} to the homogenization of $f$, we may assume that $f$ is absolutely irreducible.
For each natural number $H$, consider the  polynomial $F_H\in\ZZ[x_1,...,x_{n+2}]$ given by $F_H(x_1,...,x_{n+2})=\sum_{i=0}^d H^if_ix_{n+2}^{d-i}$. Then $F_H$ is an irreducible  homogeneous polynomial of degree $d$. On the other hand, each integral point $(x_1,...,x_{n+1})\in Z(f)(\ZZ)$ gives us a rational point $(x_1,...,x_{n+1},H)$ in $Z(F_H)(\QQ)$, where $Z(f)$ stands for the hypersurface in $\AA^{n+1}$ given by $f$ and $Z(F_H)$ stands for the hypersurface in $\PP^{n+1}$ given by $F_H$.

If $B$ is bounded by some polynomial expression in $d$ (to be determined later), then $B^{\frac{1}{nd^{1/n}}}$ is bounded by a constant depending only on $n$; hence we use Theorem \ref{thm:walsh-hypers-precise} for $F_1$, by which there exists a number $c$ depending only on $n$ along with a homogeneous polynomial $G_1$ in $\ZZ[x_1,\ldots,x_{n+2}]$ of degree at most
\[
c B^{\frac{1}{d^{1/n}}} d^{4-1/n} \frac{b(F_1)}{\lVert F_1\rVert^{\frac{1}{n} \frac{1}{d^{1+1/n}}}}  +   c d^{1-1/n} \log B + cd^{4-1/n},
\]
not divisible by $F_1$, and vanishing at all points on $Z(F_1)(\QQ)$ of height at most $B$. Since $b(F_1) = b(f)$ and $\lVert F_1\rVert \geq \lVert f_d \rVert$, by Lemma \ref{lem:ellenberg_estimate} we obtain
 \[  d^{4-1/n} \frac{b(F_1)}{\lVert F_1\rVert^{\frac{1}{n} \frac{1}{d^{1+1/n}}}} \leq O_n(d^{2-1/n}) \frac{\min(d^2 b(f), \log\lVert f\rVert + d^2)}{\lVert f_d\rVert^{\frac{1}{n} \frac{1}{d^{1+1/n}}}} .\]
Hence the polynomial $g(x_1,...,x_{n+1})=G_1(x_1,...,x_{n+1},1)$ satisfies our proposition.

For any $B \geq 2$ Bertrand's postulate guarantees the existence of a prime $B'$ in the interval $(\frac{B}{2},B]$. Moreover, if $B'\nmid f_0$, then $F_{B'}$ is primitive. By Theorem \ref{thm:walsh-hypers-precise} for $F_{B'}$, there exists a number $c$ depending only on $n$ along with a homogeneous polynomial $G_{B'}$ in $\ZZ[x_1,\ldots,x_{n+2}]$ of degree
at most \[
c B^{\frac{n+1}{nd^{1/n}}} d^{4-1/n} \frac{b(F_{B'})}{\lVert F_{B'} \rVert^{\frac{1}{n} \frac{1}{d^{1+1/n}}}}  +   c d^{1-1/n} \log B + cd^{4-1/n},
\]
not divisible by $F_{B'}$, and vanishing at all points on $Z(F_{B'})(\QQ)$ of height at most $B$.

It is clear that $\lVert F_{B'}\rVert\geq B'^d \lVert f_d \rVert \geq 2^{-d}B^{d} \lVert f_d \rVert$, so by Lemma \ref{lem:ellenberg_estimate} we have
\[ d^{4-1/n} \frac{b(F_{B'})}{\lVert F_{B'}\rVert^{\frac{1}{n} \frac{1}{d^{1+1/n}}}} \leq O_n(1) (\frac{B}{2})^{-\frac{1}{nd^{1/n}}} d^{2-1/n} \frac{\log \lVert f_d\rVert + d \log B + d^2}{\lVert f_d \rVert^{\frac{1}{n} \frac{1}{d^{1+1/n}}}}. \]
Furthermore $b(F_{B'})$ agrees with $b(F_1)$ up to a factor of $\exp(\log B'/B') \leq O(1)$. Hence we in fact have
\[  d^{4-1/n} \frac{b(F_{B'})}{\lVert F_{B'} \rVert^{\frac{1}{n} \frac{1}{d^{1+1/n}}}} \leq
O_n(1) B^{-\frac{1}{nd^{1/n}}} d^{2-1/n} \frac{\min(\log \lVert f_d\rVert + d \log B + d^2, b(f))}{\lVert f_d \rVert^{\frac{1}{n} \frac{1}{d^{1+1/n}}}}. \]

Thus the polynomial $g(x_1,...,x_{n+1})=G_{B'}(x_1,...,x_{n+1},B')$ is as desired.

From now on, we suppose that $B>2$ and $B'|f_0$ for all primes $B'$ in the interval $(\frac{B}{2},B]$. Then we have
$$(\prod_{\stackrel{B' \text{ prime}}{B/2<B'\leq B}}B')|f_0$$

If $f_0\neq 0$ then we deduce that
$$\sum_{B'{\rm prime},\ B/2<B'\leq B}\log B'\leq \log|f_0| .$$

By Lemma \ref{lem:bounded_coefficients_or_vanishing_g}, we are done if $f_0$ is large compared to $B^{d \binom{d+n+1}{n+1}}$, so in the remaining case we have
\[ \sum_{B'{\rm prime}, B/2<B'\leq B}\log B'\leq d \binom{d+n+1}{n+1} \log B - O_n(1) \]

Because of the well-known estimate \[ \lim_{x\mapsto+\infty}\dfrac{\sum_{p\leq x}\log p}{x}=1 ,\]
we see that $B$ is necessarily bounded by a certain polynomial in $d$ in this case, so we are done by the discussion above.

%
%Since we have
%$$\lim_{x\mapsto+\infty}\dfrac{\sum_{p\leq x}\log p}{x}=1$$
%So there exists constant $C_n$ such that $|x-\sum_{p\leq x}\log p|\leq \dfrac{x}{6n}+\dfrac{C_n}{2}$ for all $x\geq 1$. Therefore we have
%$$|\sum_{B'{\rm prime}, B/2<B'\leq B}\log B'-\dfrac{B}{2}|\leq \dfrac{B}{4n}+C_n$$
%Hence
%$$|\sum_{B'{\rm prime}, B/2<B'\leq B}\log B'|\geq (\dfrac{1}{2}-\dfrac{1}{4n})B-C_n$$
%As a consequence, we have $\log|f_0|\geq (\dfrac{1}{2}-\dfrac{1}{4n})B-C_n$
%On the other hand, if there exists a homogeneous polynomial $G_1\in\ZZ[x_1,...,x_{n+2}]$ of degree $d$ such that $G_1$ vanish on all points $x\in \ZZ^{n+2}$ satisfying $F_1(x)=0$ and $H(x)\leq B$
%then we can finish our proof. Otherwise, lemma \ref{boundcoef} tell us that
%$$\log|f_0|\leq \log||f||=\log||F_1||\leq \log 2+ ({d+n+1 \choose n+1}-1)(d\log B+\log {d+n+1 \choose n+1})$$
%Since $\log|f_0|\geq (\dfrac{1}{2}-\dfrac{1}{4n})B-C_n$, there exists a constant $c_n$ depended on $n$ such that $B\leq c_n d^{n+3}$; in particular $B^{1/d^{1/n}}$ is bounded by a constant depending only on $n$ in this situation.
%We can use Theorem \ref{thm:walsh-hypers} for $F_1$ to see that there exists a constant $c$ dependent only on $n$ and a homogeneous polynomial $G_{1}\in \ZZ[x_1,...,x_{n+2}]$ of degree at most
%\[
%c B^{\frac{1}{d^{1/n}}} \frac{d^{2-1/n} \log ||f_d||  + d^{4-1/n}}{||f_d||^{\frac{1}{n} \frac{1}{d^{1+1/n}}}}  +   c d^{1-1/n} \log B + cd^{n+1},
%\]
%not divisible by $F_1$, and vanishing at all points on $Z(F_1)(\QQ)$ of height at most $B$, so polynomial $g(x_1,...,x_{n+1})=G_1(x_1,...,x_{n+1},1)$ satisfies our proposition.

If $f(0)=0$, then by Corollary \ref{cor:find_point} there exists an integer point $A=(a_1,...,a_{n+1})$ with $f(a_1,...,a_{n+1})\neq 0$ and $|a_i|\leq d$ for all $1\leq i\leq n+1$. We consider the shifted polynomial $\tilde{f}(x)=f(x+A)$, for which $\tilde{f}(0)\neq 0$, $\lVert f_d \rVert = \lVert \tilde{f}_d \rVert$, and $b(\tilde{f}) = b(f)$. We apply the above discussion for $\tilde{f}$ and $\tilde{B}=B+d$ to obtain a polynomial $\tilde{g}(x)$ vanishing on all zeroes of $\tilde{f}$ of height at most $\tilde{B}$, and take $g(x) = \tilde{g}(x-A)$.
This satisfies the required degree bound since $\tilde{g}$ does.
%so there exist constant $c$ and polynomial $\tilde{g}\in\ZZ[x_1,...,x_{n+1}]$ of degree at most
%\[
%c (B+d)^{\frac{1}{d^{1/n}}} \frac{d^{2-1/n} (\log ||f_d||+d\log (B+d))  + d^{4-1/n}}{||f_d||^{\frac{1}{n} \frac{1}{d^{1+1/n}}}}  +   c d^{1-1/n} \log (B+d) + cd^{n+1},
%\]
%not divisible by $\tilde{f}$ and vanishing on all points $x$ in $\ZZ^{n+1}$ satisfying $\tilde{f}(x)=0$ and $H(x)\leq B+d$. Because $2d<B$ then $B+d\leq 2B$, so we can enlarge $c$ to see that degree of $\tilde{g}$ is at most
%\[
%c B^{\frac{1}{d^{1/n}}} \frac{d^{2-1/n} (\log ||f_d||+d\log B)  + d^{4-1/n}}{||f_d||^{\frac{1}{n} \frac{1}{d^{1+1/n}}}}  +   c d^{1-1/n} \log B + cd^{n+1},
%\]
%Then $g(x)=\tilde{g}(x-A)$ is a polynomial as we want.
\end{proof}

\begin{cor}\label{cor:Ellenb}
There exists a constant $c$ %and $e$,
such that for all $d>0$ and all irreducible affine curves $X\subseteq \AA_\QQ^2$ of degree $d$, cut out by an irreducible primitive polynomial $f\in\ZZ[x_1,x_2]$, and all $B\geq 1$ one has
$$
N_{\rm aff}(X,B) \leq c  B^{1/d}   \frac{\min(d^2\log \lVert f_d \rVert + d^3\log B + d^4, d^4 b(f))}{\lVert f_d \rVert^{1/d^2}} + cd \log B + cd^4  .
$$
%raf: not needed anymore:
%Moreover, if $d=2$, then there is $c$ such that for all $B\geq 1$
%$$
%n(X,B) \leq c  B^{1/2}.
%$$
%\change{Moreover, one can take $e= ... $.}
\end{cor}
\begin{proof}
Take $n=1$ in Proposition \ref{prop:Ellenb} and apply B\'ezout's theorem.
%raf: not needed anymore:
%The statement for $d=2$ follows from \cite[page 6]{Serre-Mordell}.
\end{proof}

If the absolute irreducibility of $f$ can be explained by the indecomposability of its Newton polytope, e.g., in the sense of~\cite{gao}, then this allows for good bounds on $b(f)$ which get rid of the factor $\log B$. The following instance will be used to prove Theorem~\ref{thm:bhargava_and_co}:

\begin{cor} \label{cor:indecomposableNP}
There exists a constant $c$ such that for all affine curves $X \subseteq \AA^2_\QQ$
cut out by a polynomial $f \in \ZZ[x_1, x_2]$ of the form
\[  c_dx_1^d + c_{d'}x_2^{d'} + \sum_{\substack{i, i' \\ id' + i'd < dd'}} c_{ij}x_1^ix_2^{i'}  \]
with $d > d' > 0$ coprime integers and $c_d, c_{d'} \neq 0$, and for all $B \geq 1$,
one has
\[ N_{\rm aff}(X,B)\leq c d^4  (\log|c_dc_{d'}| + 1)B^{1/d}.\]
\end{cor}
\begin{proof}
  By dividing out by the greatest common divisor of the coefficients, we may suppose that $f$ is primitive.
 The presence of the edge $(d,0)$--$(0,d')$ in the Newton polytope of $f$ is enough to guarantee absolute irreducibility in any characteristic~\cite[Theorem~4.11]{gao}. Therefore we can bound
 \[ b(f) \leq \prod_{p \mid c_dc_{d'}} \exp ( \frac{\log p}{p} ) \leq \log |c_d c_{d'}| + 1 \]
 through Mertens' first theorem as in Corollary \ref{cor:estimate_bad_primes}.
\end{proof}

\subsection{Proofs of our main results}

We can now prove our main theorems, subject to the following propositions; they allow us to reduce to the case of hypersurfaces throughout, and will be established in Section \ref{sec8} by projection arguments.
%\begin{prop}\label{prop:affine_reduction_plane_curve}
%  Given a curve $C$ in $\AA^n$ of degree $d$, there exists a curve $C'$ in $\AA^2$ birational to $C$, also of degree $d$, such that for any $B \geq 1$ we have $N_{\rm aff}(C,B) \leq N_{\rm aff}(C', c_n d^{e_n} B) + d^2$, where $c_n, e_n$ are constants depending only on $n$.
%\end{prop}

\begin{prop}\label{prop:affine_reduction_projection}
  Given a geometrically integral affine variety $X$ in $\AA^n$ of dimension $m$ and degree $d$, there exists a geometrically integral affine variety $X'$ in $\AA^{m+1}$ birational to $X$, also of degree $d$, such that for any $B \geq 1$ we have \[ N_{\rm aff}(X,B) \leq d N_{\rm aff}(X', c_n d^{e_n} B), \] where $c_n, e_n$ are constants depending only on $n$.

  For $m=1$, we can even achieve \[ N_{\rm aff}(X,B) \leq N_{\rm aff}(X', c_n d^{e_n} B) + d^2 .\]
\end{prop}

\begin{prop}\label{prop:projective_reduction_projection}
  Given a geometrically integral projective variety $X$ in $\PP^n$ of dimension $m$ and degree $d$, there exists a geometrically integral projective variety $X'$ in $\PP^{m+1}$ birational to $X$, also of degree $d$, such that for any $B \geq 1$ we have \[ N(X,B) \leq d N(X', c_n d^{e_n} B), \] where $c_n, e_n$ are constants depending only on $n$.

  For $m=1$, we can even achieve \[ N(X,B) \leq N(X', c_n d^{e_n} B) + d^2 .\]
\end{prop}

\begin{proof}[Proof of Theorem \ref{thm:W1.1}]
  In the case of a planar curve, i.e.~for $n=2$, Corollary \ref{cor:precise_Walsh_planar_curve} gives the claim.
%  and the trivial bounds from Section \ref{sec:triv} give
%  \[ N(X,B) \leq O_n(d^4 B^{2/d}), \]
%  using that
%  \[ \frac{d^4 b(f)}{\lVert f\rVert^{1/d^2}} \leq \frac{d^2 \log \lVert f\rVert + d^4}{\lVert f\rVert^{1/d^2}} \leq 2d^4\]
%  by Corollary \ref{cor:estimate_bad_primes}.
  For the general case, we may assume that the given curve is geometrically integral by Remark \ref{rem:wlog_abs_irred}, and then reduce to $n=2$ by applying Proposition \ref{prop:projective_reduction_projection} (where $m=1$).
\end{proof}
\begin{proof}[Proof of Theorem \ref{thm:Bombieri-Pila}]
We may assume that the curve $X$ is geometrically integral by Remark \ref{rem:wlog_abs_irred}.
In the case of a planar curve, i.e.~for $n=2$, Corollary \ref{cor:Ellenb} yields that
  \[ N(X,B) \leq O_n((d^3 \log B + d^4)B^{1/d}) ,\]
  by observing that
  \[  \frac{d^2\log \lVert f_d \rVert + d^3\log B + d^4}{\lVert f_d \rVert^{1/d^2}} \leq d^3 \log B + 2d^4 .\]
We can reduce the general case to $n=2$ by applying Proposition \ref{prop:affine_reduction_projection} (where $m=1$), yielding the same estimate.
\end{proof}

\begin{proof}[Proof of Theorem~\ref{thm:bhargava_and_co}]
  In the penultimate step of their proof of \cite[Theorem~1.1]{BSTTTZ}, Bhargava et al.\ establish the bound
  \[ h_2(K) \leq O_{d,\varepsilon} (|\Delta_K|^{\frac{1}{4} + \varepsilon}) + \sum_{\beta \in \mathcal{B}} N_{\rm aff}(f_\beta,|\Delta_K|^{\frac{1}{2}}) \]
  where $\mathcal{B} \subseteq \mathcal{O}_K$ is a set of size $O_d(|\Delta_K|^{\frac{1}{2} - \frac{1}{d}})$ and \[ f_\beta = y^2 - N_{K/\QQ}(x - \beta) = y^2 - x^d - \text{lower order terms in $x$}. \] Theorem~\ref{thm:Bombieri-Pila} implies that
  \[ N_{\rm aff}(f_\beta,|\Delta_K|^{\frac{1}{2}}) \leq O_d(|\Delta_K|^{\frac{1}{2d}} \log |\Delta_K|) ,\]
  yielding the desired result when $d$ is even.
   If $d$ is odd then instead of Theorem~\ref{thm:Bombieri-Pila} we apply Corollary~\ref{cor:indecomposableNP} with $d'=2$, $c_d = -1$, $c_{d'} = 1$ to get rid of the factor $\log | \Delta_K|$.
\end{proof}

For the proof of Theorem \ref{thm:0.4}, we need the following explicit form of Proposition 1 of \cite{Brow-Heath-Salb} with $D=1$.
\begin{prop}\label{prop1:Brow-Heath-Salb}
There exists a constant $c$
such that for all $d\geq 3$ and all polynomials $f\in \ZZ[x_1,x_2,x_3]$ of degree $d$ such that the highest degree part $h(f)=f_d$ of $f$ is  irreducible, all finite sets $I$ of curves $C$ of $\AA^3_\QQ$ of degree $1$ and lying on the hypersurface defined by $f$, and all $B\geq 1$ one has
$$
N_{\rm aff}(X  \cap ( \cup_{C \in I}  C   )  ,B) \leq c  d^6 B + \# I. %        \frac{(d^2\log ||f_d||+ d^3\log B + d^4)}{||f_d||^{1/d^2}} + cd^3  .
$$
\end{prop}
\begin{proof}
We write $I=I_{1}\cup I_2$ where $I_1=\{L\in I| N_{\rm aff}(L,B)\leq 1\}$ and $I_2=\{L\in I|N_{\rm aff}(L,B)>1\}$.
It is clear that $N_{\rm aff}(X\cap\cup_{L\in I_1}L)\leq \#I_1$. If $L\in I_2$, then there exist $a=(a_1,a_2,a_3), v=(v_1,v_2,v_3)\in \ZZ^{3}$ such that $H(a)\leq B$, $v$ is primitive and $L(\QQ)=\{a+\lambda v|\lambda\in\QQ\}$.
Since $v$ is primitive we deduce that
$$
L(\ZZ)\cap [-B,B]^{3}=\{a+\lambda v|\lambda\in \ZZ,\  H(a+\lambda v)\leq B\}.
$$
So
\[ \#(L(\ZZ)\cap [-B,B]^{3}) \leq 1+\dfrac{2B}{H(v)}.\] Since $L\in I_2$ we have $H(v)\leq 2B$ and $f_d(v)=0$. On the other hand, for each point $v$ with $f_d(v)=0$, there are at most $d(d-1)$ lines $L\in I_2$ in the direction of $v$, since each such line intersects a generic hyperplane in $\AA^3$ in a point which is simultaneously a zero of $f$ and of the directional derivative of $f$ in the direction of $v$.
Put $A_i=\{v\in \PP^{2}(\QQ)|f_d(v)=0,\  H(v)=i\}$ and $n_i=\#A_i$. Then, by Corollary \ref{cor:precise_Walsh_planar_curve}, there exists a constant $c$ independent of $f$ such that $\sum_{1\leq i\leq k}n_i\leq cd^{4}k^{\frac{2}{d}}$. By our discussion,
$$
N_{\rm aff}(X  \cap ( \cup_{C \in I}  C   )  ,B)\leq \#I_1+ (d-1)d\sum_{i=1}^{2B}n_i(1+\dfrac{2B}{i}) .
$$
On the other hand, summation by parts gives the following:
\begin{align*}
\sum_{i=1}^{2B}n_i(1+\dfrac{2B}{i})&=\sum_{k=1}^{2B-1}(\sum_{i=1}^{k}n_i)(\dfrac{2B}{k}-\dfrac{2B}{k+1})+(\sum_{i=1}^{2B}n_i)(1+\dfrac{2B}{2B})\\
&\leq cd^{4}(\sum_{k=1}^{2B-1}k^{\frac{2}{d}}\dfrac{2B}{k(k+1)}+2(2B)^{\frac{2}{d}}).
\end{align*}
Since $d\geq 3$, one has $\sum_{k\geq 1}k^{\frac{2}{d}}\frac{1}{k(k+1)}<+\infty$ and $B^{\frac{2}{d}}\leq B$. Thus, by enlarging $c$, we have
$$
N_{\rm aff}(X  \cap ( \cup_{C\in I}  C   )  ,B)\leq cd^{6}B+\#I
$$
as desired.
\end{proof}

In order to prove Theorem \ref{thm:0.4}, we now first consider the case of a surface in $\PP^3$, with proof inspired by the proof of Corollary 7.3 of \cite{Salberger-dgc} in combination with the improvements developed above.
\begin{prop}\label{prop:0.4strong}
There exists a constant $c$ such that for all polynomials $f$ in $\ZZ[y_1,y_2,y_3]$ whose homogeneous part of highest degree $f_d$ is irreducible over $\overline\QQ$ and whose degree $d$ is least $5$, one has
$
N_{\rm aff}(f,B)\leq c d^{14} B
$.
\end{prop}
\begin{proof}[Proof of Proposition \ref{prop:0.4strong} for $d\geq 16$]
For any prime modulo which $f_d$ is absolutely irreducible, the reduction of $f$ is likewise absolutely irreducible, so $b(f) \leq b(f_d)$.
Applying the usual estimate from Corollary \ref{cor:estimate_bad_primes}, Proposition \ref{prop:Ellenb} yields for each $B \geq 1$ a polynomial
$g$ of degree at most
\begin{equation}\label{eq:degree:g}
cd^{7/2} B^{1/\sqrt{d}},
\end{equation}
not divisible by $f$ and vanishing on all points  $x$ in $\ZZ^n$ satisfying $f(x)=0$ and $|x_i|\leq B$, with $c$ an absolute constant. Let $C$ be an irreducible component of the (reduced) intersection of $f=0$ with $g=0$. Call this intersection $\cC$.
%We first prove that
%$$
%n(f,B) %:= \#\{x\in \ZZ^n\mid |x_i|< B \mbox{ for each $i$ and }f(x)=0\}
%\leq c d^e \big( B^{ 2/\sqrt{d}} (\log B  + 1 ) + B\big),
%$$
%from which the cases $d=3$ a
If $C$ is of degree $\delta>1$, then
\begin{equation}\label{eq:delta}
N_{\rm aff}(C,B) \leq c'  \delta^{3} B^{1/\delta} (\log B +\delta)
\end{equation}
by Theorem \ref{thm:Bombieri-Pila}, for some absolute constant $c'$.%; in fact we can take $e' < \sqrt{17}$.

%For each $\delta >1$, there can at most be
%$$
%c\frac{d^{e+1}}{\delta} B^{1/\sqrt{d}} (\log B  + 1 )
%$$
%many irreducible components  of $\cC$ of degree $\delta$.

%Hence, the total contribution to $n(f,B)$ of all irreducible components $C$ of $\cC$ of degree exactly $\delta$ with $\delta>1$ is at most
%$$
%cc'  \delta^{e'-1} B^{1/\delta} (\log B +1) d^{e+1} B^{1/\sqrt{d}} (\log B  + 1 )
%$$

%larger than $2$ is at most
%$$
%cd^e B^{\frac{1}{\sqrt{d}} + \frac{1}{3} } (\log B  + 1 ),
%$$
%for some $c,e$, independent from $d$ and $B$.
By Proposition \ref{prop1:Brow-Heath-Salb}, the total contribution of integral curves $D$ of $\cC$ of degree  $1$ is at most
\begin{equation}\label{eq:degree1}
c''d^6 B
\end{equation}
for some absolute constant $c''$.
%Furthermore, again by Proposition \ref{prop1:Brow-Heath-Salb} %for the cases $d\not=4$ and by Proposition 1 of \cite{prop1:Brow-Heath-Salb} for the case $d=4$,
%the total contribution of all integral curves $D$ of the (reduced) intersection of $f=0$ with $g=0$ of degree  $2$ is at most
%$$
%cd^e B^{\frac{1}{\sqrt{d}} + \frac{1}{2} } (\log B  + 1 ),
%$$
%for some $c,e$, independent from $d$ and $B$.

%Now we first complete the proof when $d>4$.

Suppose that $C_1,...,C_k$ are irreducible components of the intersection of $f=0$ and $g=0$ and $\deg(C_i)>1$ for all $i$. Furthermore, we assume that $\deg(C_i)\leq\log B$ for all $1\leq i\leq m$ and $\deg(C_i)>\log B$ for all $i>m$. Since the function $\delta \mapsto 4\log_B(\delta)+\frac{1}{\delta}$ is decreasing in $(0,\frac{\log B}{4})$ and increasing in $(\frac{\log B}{4},+\infty)$, by enlarging $c'$, for all $1\leq i\leq m$ we have
\begin{equation}\label{eq:delta1}
N_{\rm aff}(C_i,B) \leq c' B^{1/2} (\log B +1).
\end{equation}
On the other hand, if $\delta> \log B$ then $B^{\frac{1}{\delta}}$ is bounded, so \eqref{eq:degree:g} and \eqref{eq:delta} imply
\begin{equation}\label{eq:delta2}
\sum_{m+1\leq i\leq k}N_{\rm aff}(C_i,B) \leq c'''d^{14}B^{4/\sqrt{d}}
\end{equation}
for some $c'''$ independent of $d$ and $B$.

Putting the estimates (\ref{eq:degree:g}),  (\ref{eq:degree1}), (\ref{eq:delta1}), (\ref{eq:delta2}) together proves the proposition when $d$ is at least $16$. %large (compared to $e'$ from (\ref{eq:delta}). %except for the case $d=4$.
%\change{make this more explicit!}
%
%Finally, when $d=3$ or when $d=4$, Proposition 1 of \cite{Brow-Heath-Salb} gives that the total contribution of all integral curves $D$ of $\cC$ of degree  $2$ is at most
%\begin{equation}\label{eq:degree2}
%c''' B
%\end{equation}
%for some $c'''$ independent from $B$.  Putting the estimates (\ref{eq:degree:g}), (\ref{eq:delta}) with $\delta>2$, (\ref{eq:degree1}), and  (\ref{eq:degree2})  finishes the cases $d=3$ and $d=4$. \change{make this more explicit!}
\end{proof}

To give a proof of Proposition \ref{prop:0.4strong} for lower values of $d$ than $16$, one could try to get a form of Theorem \ref{thm:Bombieri-Pila} with a lower  exponent of the degree and repeat the above proof. We proceed differently: we treat the values for $d$ going from $6$ up to $15$ by inspecting the proof of \cite[Theorem 2]{Brow-Heath-Salb} in combination with some of the above refinements, and the case of $d=5$ by using \cite[Theorem 7.2]{Salberger-dgc} (at the cost of being less self-contained). %and our Proposition \ref{prop1:Brow-Heath-Salb} and Theorem \ref{thm:Bombieri-Pila}.
 
\begin{proof}[Proof of Proposition \ref{prop:0.4strong} with $6 \leq d \leq 15$]
  Fix $6 \leq d \leq 15$, let $f \in \ZZ[y_1, y_2, y_3]$ be of degree $d$ with absolutely irreducible homogeneous part of highest degree, and let $X$ be the surface described by $f$.

  In \cite[Theorem 2]{Brow-Heath-Salb}, the estimate $N_{\rm aff}(f, B) \leq O_{d,\varepsilon}(B^{1+\varepsilon})$ is established for every $\varepsilon > 0$. However, using our Theorem \ref{thm:W1.1} and Proposition \ref{prop1:Brow-Heath-Salb}, we shall show that their proof \cite[pp. 568--570]{Brow-Heath-Salb} in fact gives the bound $N_{\rm aff}(f,B) \leq O_d(B)$, without any $\varepsilon$, which is sufficient for our purposes.

  Specifically, they first consider the case in which Lemma \ref{lem:bounded_coefficients_or_vanishing_g} applies, so all the rational points on $X$ of height up to $B$ lie on a union of irreducible curves with sum of degrees at most $d^2$. Applying Theorem \ref{thm:W1.1} to those curves of degree $\geq 2$ and Proposition \ref{prop1:Brow-Heath-Salb} for the contribution of curves of degree $1$ yields the claim in this case.

  In the remaining case, it is argued that there is an open subset $U \subseteq X$ (specifically consisting of those nonsingular points on $X$ which have multiplicity at most $2$ on the tangent plane section at the point) whose complement consists of $O_d(1)$ integral components of degree $O_d(1)$; by the same argument as in the preceding paragraph, the contribution of this complement is $O_d(B)$, so it suffices to estimate $N_{\rm aff}(U,B)$.

  Further, it is argued that the points on $U$ of height at most $B$ are covered by a certain collection of irreducible curves. The subcollection $I$ consisting of those curves of degree at most $2$ satisfies $\lvert I \rvert \leq O_{d, \varepsilon}(B^{2/\sqrt{d} + 2\varepsilon})$, so our Proposition \ref{prop1:Brow-Heath-Salb} and \cite[Proposition 1]{Brow-Heath-Salb} gives a contribution $O_{d,\varepsilon}(B + B^{2/\sqrt d + 3\varepsilon}) \leq O_{d}(B)$.

  The remaining curves, of which there are no more than $O_{d,\varepsilon}(B^{2/\sqrt d})$, all contribute at most $B^{1/3-1/(2\sqrt d)}$ (\cite[Proposition 2]{Brow-Heath-Salb}), so their total contribution is \[ O_{d,\varepsilon}(B^{3/(2\sqrt d) + 1/3 + \varepsilon}) \leq O_d(B). \qedhere \]
\end{proof}

\begin{thm}[{\cite[Theorem 7.2]{Salberger-dgc}}]\label{thm:7.2_salb_dgc}
  Let $X$ be a geometrically integral surface in $\PP^3_\QQ$ of degree $d$ and $X_{\mathrm{ns}}$ its non-singular locus. Suppose that the hyperplane defined by $x_0=0$ intersects $X$ properly, and identify $\AA^3$ with the open subset of $\PP^3$ given by $x_0 \neq 0$.
  There exists a positive constant $c$ bounded solely in terms of $d$ such that the following holds:
  for every $B \geq 1$ there exists a set of $O_d(B^{1/\sqrt d} \log B + 1)$ geometrically integral curves $D_\lambda$ on $X$ of degree $O_d(1)$ such that
  \[ N_{\rm aff}(X_{\mathrm{ns}} \setminus \bigcup_\lambda D_\lambda, B) \leq O_d(B^{2/\sqrt d + c/\log(1 + \log B)}). \]
\end{thm}

\begin{proof}[Proof of Proposition \ref{prop:0.4strong} for $d = 5$]
  Suppose that the degree $d$ of $f$ is exactly $5$, and let $X$ be the surface in $\AA^3_\QQ$ given by $f$. We may assume that $B \geq 2$.
  By Theorem \ref{thm:7.2_salb_dgc}, there is $c>0$ such that for each $B\geq 2$ there is a set $\mathcal{C}$ of at most
$$cB^{1/\sqrt{d}}\log B$$
geometrically integral curves $C\subseteq \AA^3_\QQ$ of degree at most $c$ and lying on $X$ such that
$$N_{\rm aff}( X_{\rm ns} \setminus \bigcup_{C\in \mathcal{C}} C,   B)\leq O(B^{2/\sqrt{d} + c/\log(\log B) }) \leq O(B^{1/2}), $$
where $X_{\rm ns}$ is the open subvariety of nonsingular points.

  The complement of $X_{\rm ns}$ in $X$ is a union of irreducible curves the sum of whose degrees is bounded by a constant.
  Applying Theorem \ref{thm:W1.1} to those curves of degree $\geq 2$ and Proposition \ref{prop1:Brow-Heath-Salb} for the contribution of curves of degree $1$ yields that the complement of $X_{\rm ns}$ contributes at most $O(B)$ points, which is satisfactory for our purposes.

  Similarly, the curves in $\mathcal{C}$ of degree $1$ contribute at most $O(cB^{1/\sqrt{d}}\log B + B) \leq O(B)$ points by Proposition \ref{prop1:Brow-Heath-Salb}, and the curves in $\mathcal{C}$ of degree $\geq 2$ each contribute at most $O(B^{1/2+\varepsilon})$ by Theorem \ref{thm:Bombieri-Pila}, again giving a contribution of size $O(B)$. This proves the claim.
\end{proof}

\begin{remark}
  We see that Proposition \ref{prop:0.4strong} for fixed $d \geq 6$, and therefore also Theorems \ref{thm:dgcdegree} and \ref{thm:0.4} for fixed degree, already follow from combining \cite{Brow-Heath-Salb} with the results of \cite{Walsh} and Proposition \ref{prop1:Brow-Heath-Salb}. Similarly, for fixed degree $d \geq 5$ one can use the results of \cite{Salberger-dgc}.
  However, keeping track of the dependence on $d$ in Section \ref{sec:Walsh} permits us to use a considerably simpler argument for fixed $d \geq 16$ than in the works cited, and to furthermore obtain polynomial dependence on $d$.
\end{remark}

%
%\begin{proof}[Proof of Theorem \ref{thm:0.4}]
%The case $n=3$ follows from Proposition \ref{prop:0.4strong}.
%%\change{For general $n$, the theorem follows from the case $n=3$ and the projection results in section \ref{sec8}.}
%%Use our version of 3.16 of \cite{Salberger-dgc} as it is used in Section 7 of \cite{Salberger-dgc} for 7.2, 7.3 and 7.4 to prove the $n=3$ case.
%%Use the reduction argument of Section 4 (Lemma 8) of  \cite{Brow-Heath-Salb}  to do the case of general $n$.
%\end{proof}

It remains to prove Theorems \ref{thm:dgcdegree} and \ref{thm:0.4}. This closely follows \cite[Lemma 8, Theorem 3]{Brow-Heath-Salb}.
The proofs are based on Proposition \ref{prop:0.4strong} and the following lemma.

 \begin{lem}\label{hypersection}
Let $n\geq 3$ and $X\subseteq \PP_\QQ^{n}$ be a geometrically integral hypersurface of degree $d$. Then there exists a non-zero form $F\in \ZZ[y_0,\dotsc,y_{n}]$ of degree at most $(n+1)(d^2-1)$ such that $F(A)=0$ whenever the hyperplane section $H_{A}\cap X$ is not geometrically integral, where $A\in (\PP^n)^*$ and $H_A \subseteq \PP^n$ denotes the hyperplane cut out by the linear form associated with $A$.
\end{lem}
\begin{proof}
Suppose that $X$ is given by $f$, a geometrically irreducible form of degree $d$. For  $A\in (\PP^{n})^{*}$ write $A=(a_0:a_1:\ldots :a_n)\in (\PP^n)^*$.
% be such that $H_{A}\cap X$ is not geometrically integral and
%. We suppose that is $A=(a_0:a_1...:a_n)\in (\PP^n)^*$ such that $H_{A}\cap X$ is not geometrically integral.
Assuming $a_0\neq 0$, one has that $H_{A}\cap X$ is not geometrically integral if and only if $$f(-\dfrac{a_1}{a_0}x_1-...-\dfrac{a_n}{a_0}x_n,x_1,\dotsc,x_n)$$ is reducible.  Since $n\geq 3$ and since $X$ is geometrically integral, we have for a generic choice of $B\in (\PP^n)^*$ that $H_{B}\cap X$ is also geometrically integral. Hence Theorem \ref{thm:explicit_noether} implies that there exists a non-zero form $F_0$ in $\ZZ[y_1,\dotsc,y_n]$ of degree at most $d^2-1$ such that $F_0(a_1,\dotsc,a_n)=0$. Similarly, if $a_i\neq 0$, we produce a non-zero form $F_i$ in $\ZZ[y_0,\dotsc,y_{i-1}, y_{i+1},\dotsc,y_n]$ such that $F_{i}(a_0,\dotsc,a_{i-1}, a_{i+1}, \dotsc,a_n)=0$. So $F=\prod_{i=0}^{n} F_i$ is as we want.
\end{proof}

\begin{proof}[Proof of Theorem \ref{thm:0.4}]
Let $n\geq 3$ and $X\subseteq \AA_\QQ^{n}$ be a geometrically integral hypersurface of degree $d\geq 5$ described by a polynomial $f \in \ZZ[x_1, \dotsc, x_n]$ with absolutely irreducible highest degree part.
We proceed by induction on $n$, where the base case $n=3$ is Proposition \ref{prop:0.4strong}.

  Now assume that $n > 3$ and the theorem holds for all lower $n$. Let $f_d = h(f)$ be the homogeneous part of highest degree, which describes a hypersurface in $\PP^{n-1}$.
  By Lemma \ref{hypersection} and Corollary \ref{cor:find_point}, there exists $A=(a_1, \dotsc, a_n)$ such that the hyperplane section $\{ f_d = 0 \} \cap \{ \sum a_i x_i = 0 \}$ is geometrically integral of degree $d$, with all $a_i$ having absolute value at most $n(d^2-1)$.

  Now \[ N_{\rm aff}(f,B) \leq \sum_{\lvert k \rvert \leq n^2 (d^2-1) B} N_{\rm aff}(\{ f = 0 \} \cap \{ \sum a_i x_i = k \}, B)  .\]
  For each $k$, the variety $\{ f = 0 \} \cap \{ \sum a_i x_i = k \}$ is a hypersurface in the affine plane $\{ \sum a_i x_i = k \}$, which after a change of variables is described by a polynomial $g \in \ZZ[x_1, \dotsc, x_{n-1}]$ whose homogeneous part of highest degree is absolutely irreducible by the construction of $A$. Now the induction hypothesis finishes the proof.
\end{proof}

\begin{proof}[Proof of Theorem \ref{thm:dgcdegree}]
  We may assume that the variety in question is geometrically irreducible by Remark \ref{rem:wlog_abs_irred}, and can reduce to consideration of a hypersurface by Proposition \ref{prop:projective_reduction_projection}.
  Hence let $n \geq 3$ and consider an absolutely irreducible polynomial $f \in \ZZ[x_0, \dotsc, x_n]$ homogeneous of degree $d \geq 5$.

  Then $f$ defines not only a projective hypersurface $X$ in $\PP^n$, but also an affine hypersurface in $\AA^{n+1}$, the cone of $X$.
  We now trivially have \[ N(f,B) \leq N_{\rm aff}(f,B), \]
  so Theorem \ref{thm:0.4} finishes the proof.
%  By Lemmas \ref{hypersection} and \ref{findpoint}, there exists $A=(a_0, \dotsc, a_n)$ such that the hyperplane section $\{ f = 0 \} \cap \{ \sum a_i x_i = 0 \}$ is geometrically integral of degree $d$, with all $a_i$ having absolute value at most $(n+1)(d^2-1)$. Now
%  \[ N(f, B) \leq \sum_{\lvert k \rvert \leq (n+1)^2 (d^2-1) B} N_{\rm aff}(\{ f = 0 \} \cap \{ \sum a_i x_i = k \}, B) .\]
%
%   For each $k$, the variety $\{ f = 0 \} \cap \{ \sum a_i x_i = k \}$ is a hypersurface in the affine plane $\{ \sum a_i x_i = k \}$, which after a change of variables is described by a polynomial $g \in \ZZ[x_1, \dotsc, x_{n-1}]$ whose homogeneous part of highest degree is absolutely irreducible by the construction of $A$. Now Theorem \ref{thm:0.4} finishes the proof.
\end{proof}

\begin{remark}\label{rem:explicit_exponents}
  Using the explicit exponents obtained in Proposition \ref{prop:0.4strong} and in the proof of Proposition \ref{prop:projective_reduction_projection} in Section \ref{sec8}, we can conservatively estimate $e(n) \leq 2n + 8$ for the exponent in Theorem \ref{thm:0.4}, and $e(n) \leq 2n^3$ for the exponent in Theorem \ref{thm:dgcdegree}.
\end{remark}

\section{Reduction to hypersurfaces via projection}
\label{sec8}

%In this section we give an effective version of projection map from \cite{Brow-Heath-Salb}.
In this section we prove Propositions \ref{prop:affine_reduction_projection} and \ref{prop:projective_reduction_projection}, which allowed us to reduce to the case of hypersurfaces in the proofs of our main theorems.
This is an elaboration of familiar projection arguments, which classically show that every variety is birational to a hypersurface, and which are used in the proofs of \cite[Theorem 1]{Brow-Heath-Salb} and \cite[Theorem A]{Pila-ast-1995}. The additional difficulty for us is that we have to keep track of the dependence on the degree of the variety throughout.
Our main auxiliary result is: 

\begin{lem} \label{lem:projection}
Given a geometrically irreducible subvariety $X \subseteq \PP^n$ of dimension $m < n-1$ and degree $d$, one can find an $(n-m-2)$-plane $\Lambda$ disjoint from $X$ and an $(m+1)$-plane $\Gamma$, both defined over $\QQ$, such that $\Lambda \cap \Gamma = \emptyset$, such that the corresponding projection map $p_{\Lambda,\Gamma} : \PP^n \setminus \Lambda \to \Gamma$ satisfies 
\begin{equation} \label{height_after_projection}
 H(p_{\Lambda,\Gamma}(P))\leq c_n d^{2(n-m-1)^2} H(P)
\end{equation}
for all $P \in \PP^n(\QQ) \setminus \Lambda$,
and such that $p_{\Lambda,\Gamma}|_X$ is birational onto its image. Here $c_n$ is an explicit constant depending only on $n$.
\end{lem}

Because $\Lambda$ is disjoint from $X$, the statement that $p_{\Lambda,\Gamma}|_X$ is birational onto its image is equivalent to saying that $p_{\Lambda,\Gamma}(X)$ is again a variety of degree $d$, see~\cite[Example 18.16]{Harris}.

In order to prove Lemma~\ref{lem:projection}, we first concentrate on finding an appropriate $\Lambda$, which we think of as living in the Grassmannian $\GG(n-m-2,n)$ consisting of all $(n-m-2)$-planes in $\PP^n$. 
It is well-known that the latter has the structure of 
an $(m+2)(n-m-1)$-dimensional 
irreducible projective variety through the
Pl\"ucker embedding
\[ P_{n-m-2,n}: \GG(n-m-2,n) \hookrightarrow \PP^\nu : \Lambda \mapsto \det(P_1,...,P_{n-m-1}),
\]
where $\nu={n+1 \choose n-m-1}-1$ and $(P_1,...,P_{n-m-1})$ is the $(n-m-1) \times (n+1)$ matrix whose rows
are coordinates for $n-m-1$ independent points $P_i \in \Lambda$. 
% It is well-known that $\GG(k,n)$ has degree 
% \[ ((k+1)(n-k))!\prod_{i=1}^{k+1}\frac{(i-1)!}{(n-k+i-1)!} \]
%and dimension $(k+1)(n-k)$.
Here and throughout this section, for a matrix $M$ whose number of rows does not exceed its number of columns, we write $\det(M)$ to denote the tuple consisting of its maximal minors, with respect to some fixed ordering.

Fixing such a $\Lambda\in \GG(n-m-2)$ and independent points $P_1,...,P_{n-m-1} \in \Lambda$, we can also consider the map
\[
\pi_{\Lambda}:\PP^n\setminus\Lambda \to \PP^{\mu} : P \mapsto \det(P,P_1,...,P_{n-m-1}), \] 
%\qquad \mu = {n+1\choose n-m}-1. \]
where $\mu={n+1 \choose n-m}-1$. Writing $\pi_\Lambda=(\pi_{\Lambda,0},...,\pi_{\Lambda,\mu})$ we see that the non-zero $\pi_{\Lambda,j}$'s can be viewed as linear forms whose coefficients are coordinates of $P_{n-m-2,n}(\Lambda)$, modulo sign flips. Note that $\pi_{\Lambda,j}(P)=0$ for all $j$ if and only if $P\in\Lambda$. In particular $\pi_\Lambda$ is well-defined and easily seen to factor as
\begin{equation} \label{factorizationofpiLambda} 
  \PP^n \setminus \Lambda \stackrel{p_{\Lambda, \Gamma}}{\longrightarrow} \Gamma \hookrightarrow \PP^\mu 
\end{equation}
for all $(m+1)$-planes $\Gamma$ such that $\Gamma\cap\Lambda=\emptyset$.

Another theoretical ingredient we need is the Chow point $F_X$ associated with an irreducible $m$-dimensional degree $d$ variety $X \subseteq \PP^n$. This is an irreducible multihomogeneous polynomial of multidegree $(d,d,\ldots ,d)$ in $m+1$ sets of $n+1$ variables such that for all tuples $(H_1,H_2, \ldots ,H_{m+1})$ of $m+1$ hyperplanes in $\PP^n$ one has $F_X(H_1,\ldots,H_{m+1})=0$ if and only if $H_1\cap H_2 \cap \ldots \cap H_{m+1}\cap X\neq\emptyset$. See e.g.~\cite[Chapter 4]{GKZ}.

\begin{lem}\label{empty} Let $X$ be a geometrically irreducible degree $d$ subvariety of $\PP^n$ having dimension $m<n-1$ and consider 
%Let $Y_X$ be as in Lemma~\ref{non-empty} and define 
\[ G_{X}=\{ \, \Lambda\in \GG(n-m-2,n) \, | \, \Lambda \cap X=\emptyset \text{ and } \pi_{\Lambda}|_{X} \text{ is birational onto its image} \, \} \] 
with
$\pi_{\Lambda}$ as above. 
This is a dense open subset of $\GG(n-m-2,n)$ whose complement, when viewed under the Pl\"ucker embedding, is cut out by
hypersurfaces of degree less than $(m+1)^2d^2$.
\end{lem}

\begin{proof}
Given a hyperplane $H \subseteq \PP^{\mu}$ we
abusively write $H \circ \pi_\Lambda$ for $\pi_\Lambda^{-1}(H) \cup \Lambda$, since this is the hyperplane in $\PP^n$ cut out by 
the precomposition of $\pi_\Lambda$ with the linear form associated with $H$.
Define a multihomogeneous degree $(d,d, \ldots, d)$ polynomial $R_{X, \Lambda}$ in $m+1$ sets of $\mu + 1$ variables
by letting
\[ R_{X, \Lambda}(H_1, H_2, \ldots, H_{m+1}) = F_X(H_1 \circ \pi_\Lambda, H_2 \circ \pi_\Lambda, \ldots, H_{m+1} \circ \pi_\Lambda). \]
Note that its coefficients are degree $(m+1)d$ polynomial expressions in the coordinates of
$P_{n-m-2,n}(\Lambda)$. We will show that
\begin{equation} \label{interpretation_of_GX}
 G_{X}=\{ \, \Lambda\in \GG(n-m-2,n) \, | \, R_{X, \Lambda} \text{ is absolutely irreducible} \, \}, 
\end{equation}
which implies that the complement of $G_X$ is precisely the vanishing locus of the Noether irreducibility polynomials from Theorem~\ref{thm:explicit_noether} evaluated in these coefficients. This indeed yields expressions in the coordinates of
$P_{n-m-2,n}(\Lambda)$ of degree less than $(m+1)^2d^2$, where
we note that not all these expressions can vanish identically, since
generic $\Lambda$'s do not meet $X$ and generic projections are known to be birational~\cite[p.\,224]{Harris}.

We now prove~\eqref{interpretation_of_GX}. First note that $\Lambda \cap X \neq \emptyset$ implies that $R_{X, \Lambda}$ vanishes identically. Indeed, if $P \in \Lambda$ then all hyperplanes of the form $H \circ \pi_\Lambda$ pass through $P$, so if moreover $P \in X$ we see that
$R_{X, \Lambda}$ is identically zero. We can therefore assume that $\Lambda \cap X = \emptyset$. This ensures that $\pi_\Lambda(X)$ is an irreducible projective variety of dimension $m$, see~\cite[p.\,134]{Harris}, so we can consider
its Chow point $F_ {\pi_\Lambda(X)}$, which is an irreducible multihomogeneous polynomial of multidegree 
\[ \left( \deg (\pi_\Lambda(X)), \deg (\pi_\Lambda(X)), \ldots, \deg (\pi_\Lambda(X)) \right) \]
in the same $m+1$ sets of $\mu + 1$ variables as in the case of $R_{X, \Lambda}$. It has the property that
for all tuples $(H_1,\ldots ,H_{m+1})$ of hyperplanes in $\PP^\mu$ we have $F_ {\pi_\Lambda(X)}(H_1, \ldots ,H_{m+1})=0$ if and only if $H_1\cap \ldots \cap H_{m+1}\cap \pi_\Lambda(X) \neq\emptyset$. But in this case 
$\pi^{-1}_\Lambda(H_1) \cap \ldots \cap \pi^{-1}_\Lambda(H_{m+1}) \cap X \neq \emptyset$ so that $R_{X, \Lambda}(H_1, \ldots ,H_{m+1})=0$. Conversely, if $R_{X, \Lambda}(H_1, \ldots ,H_{m+1})=0$
then there exists a point $P \in H_1\circ\pi_\Lambda\cap \ldots \cap H_{m+1}\circ\pi_\Lambda\cap X$, which since $\Lambda \cap X=\emptyset$ implies that $\pi_{\Lambda}(P)\in H_1 \cap \ldots \cap H_{m+1} \cap \pi_\Lambda(X)$ and hence that 
$F_ {\pi_\Lambda(X)}(H_1, \ldots ,H_{m+1})=0$. We conclude that $F_{\pi_\Lambda(X)}$ and $R_{X, \Lambda}$ have the same vanishing locus and because the former polynomial is irreducible there must exist some $r \geq 1$ such that
\[ R_{X, \Lambda} = F_{\pi_\Lambda(X)}^r. \]
In particular $R_{X, \Lambda}$ is irreducible if and only if $r = 1$. But this is true if and only if $\pi_\Lambda(X)$ has degree $d$, which as we know holds if and only if $\pi_{\Lambda}|_{X}$ is birational onto its image. 
\end{proof}

\begin{lem}\label{Existence} Using the assumptions and notation from Lemma~\ref{empty}, there exists an
$(n-m-2)$-plane $\Lambda \in G_X(\QQ)$ such that 
\[ H(\Lambda) \leq ((m+1)^2d^2)^{n-m-1}(n-m-1)! \]
when considered under the Pl\"ucker embedding.
\end{lem}
\begin{proof} %Let $\nu = {n + 1 \choose n-m-1} - 1$ and 
Consider the rational map 
\[ \pi:(\PP^n)^{n-m-1} \dashrightarrow \PP^{\nu} : (P_1,\ldots,P_{n-m-1}) \mapsto \det(P_1,\ldots,P_{n-m-1}) \]
which is well-defined on the open $U$ consisting of tuples of independent points. Observe that $\pi(U) = \GG(n-m-2,n)$.
By Lemma~\ref{empty} there exists a polynomial $F$ of degree less than $(m+1)^2d^2$ which vanishes on the complement of $G_X$ but which does not vanish identically on $\GG(n-m-2,n)$. 
The polynomial
\[ Q := F \left( \det \begin{pmatrix} x_{10} & x_{11} & \ldots & x_{1n} \\
                                       x_{20} & x_{21} & \ldots & x_{2n} \\
                                        \vdots & \vdots & \ddots & \vdots \\
                                       x_{n-m-1,0} & x_{n-m-1,1} & \ldots & x_{n-m-1,n} \\  \end{pmatrix} \right) \]
is multihomogeneous of multidegree $(\deg(F), \ldots ,\deg(F))$ in the $n-m-1$ blocks of $n+1$ variables corresponding to the rows of the displayed matrix. Clearly $Q$ vanishes on the complement of $U$, while it is 
not identically zero because
$Q(P_1,\ldots,P_{n-m-1}) = F(\pi(P_1,\ldots,P_{n-m-1}))$ for any tuple of independent points $P_i$. 

Write 
\[ Q = \sum_j Q_j(x_{10}, \ldots, x_{1n}) R_j(x_{20}, \ldots, x_{n-m-1,n}) \]
for non-zero $Q_j$ and linearly independent polynomials $R_j$. Lemma~\ref{lem:gen_Schwarz_Zippel} helps us to find a point $P_1 \in\PP^n(\QQ)$ of height at most $\deg(F)$ such that $Q_1(P_1) \neq 0$.
By the linear independence of the $R_j$'s one sees that $Q(P_1,x_{20},\ldots,x_{n-m-1,n})$ is not identically zero. Repeating the argument eventually yields a tuple of points $P_1, P_2, \ldots, P_{n-m-1}$
of height at most $\deg(F)$ such that $Q(P_1,\ldots,P_{n-m-1}) \neq 0$. In particular this tuple of points belongs to $U$, i.e., they are independent, and $\pi(P_1,P_2,\ldots,P_{n-m-1})\in G_X(\QQ)$. From this the lemma follows easily.
\end{proof}

\begin{proof}[Proof of Lemma~\ref{lem:projection}]
Let $\Lambda$ be the $\QQ$-rational $(n-m-2)$-plane produced by the proof of Lemma~\ref{Existence}. 
In particular $\Lambda \cap X = \emptyset$ and $\pi_\Lambda|_X$ is birational onto its image. 
Then for all $(m+1)$-planes $\Gamma$ such that $\Gamma\cap\Lambda=\emptyset$ the projection map
$p_{\Lambda,\Gamma}|_X$ is also birational onto its image, thanks to the factorization from~\eqref{factorizationofpiLambda}.

The proof of Lemma~\ref{Existence} moreover shows that $\Lambda$ can be assumed to be the linear span of rational points $P_1, \ldots ,P_{n-m-1} \in\PP^n$ satisfying $H(P_i)\leq (m+1)^2d^2 =: B_1$. 
By Lemma~\ref{bombierivaaler2} below we can find linear forms $L_1, L_2, \ldots, L_{n-m-1}$ with integral coefficients whose absolute value is bounded by
%\[ B_2 := (n-m-2)!^{\frac{1}{2(m+2)}}\left(4nB_1^4 \right)^\frac{n-m-2}{2(m + 2)}, \] 
\[ B_2 := \sqrt{(n-m-2)! (n+1)} B_1^{n-m-2} \]
such that $L_i$ vanishes on $P_1, \ldots, P_{i-1}, P_{i+1}, \ldots, P_{n-m-1}$ but not on $P_i$.
Together these linear forms cut out an $(m+1)$-plane $\Gamma$ such that $\Gamma \cap \Lambda = \emptyset$.
Furthermore 
\begin{equation} \label{projection_formula}
p_{\Lambda, \Gamma}(P)=P-\dfrac{L_1(P)}{L_1(P_1)}P_1 - \ldots - \dfrac{L_{n-m-1}(P)}{L_{n-m-1}(P_{n-m-1})}P_{n-m-1}
\end{equation}
for all $P \in \PP^n \setminus \Lambda$. So we have 
\begin{equation} \label{projectionlemmabound} 
 H(p_{\Lambda, \Gamma}(P))\leq (n-m)((n+1)B_1B_2)^{n-m-1} H(P) = c d^{2(n-m-1)^2} H(P) 
\end{equation}
for
some constant $c$ that is easily bounded by an expression purely in $n$.
\end{proof}

\begin{lem} \label{bombierivaaler2}
Let $B, r, s \in \ZZ_{\geq 1}$ be integers such that $s < r$. Consider a linear system of linearly independent equations $\sum_{k=1}^r a_{ik} x_k  = 0$
for $i=1, \ldots, s$, where all $a_{ij}$ are integers satisfying $|a_{ij}| \leq B$. There exists a non-zero tuple of integers $(x_1, x_2, \ldots, x_r)$ violating the first equation but satisfying all other equations such that 
\begin{equation} \label{BVbound}
% |x_i| \leq (s-1)!^{\frac{1}{2(r-s)}} \left(4(r-1) B^4 \right)^\frac{s-1}{2(r-s)} 
 |x_i| \leq \sqrt{(s-1)! r} B^{s-1} 
\end{equation}
for all $i$.
\end{lem}
\begin{proof}
This follows from~\cite[Theorem 2]{Bombi-Vaal}, which strengthens~Theorem~\ref{thm:bombval_siegel}.
It ensures the existence of $r-s+1$ linearly independent 
tuples of integers $(x_1, x_2, \ldots, x_r)$ satisfying the last $s-1$ equations 
and meeting the bound~\eqref{BVbound}. Since the space of solutions to the full linear system of $s$ equations has dimension $r-s$, at least one of these tuples must violate the first equation.
\end{proof}

We can now prove Propositions \ref{prop:affine_reduction_projection} and \ref{prop:projective_reduction_projection}, reducing the situation of a general variety to a hypersurface.
%\begin{prop}\label{prop:affine_reduction_projection}
%  Given an integral affine variety $X$ in $\AA^n$ of dimension $m$ and degree $d$, there exists an integral affine variety $X'$ in $\AA^{m+1}$ birational to $X$, also of degree $d$, such that for any $B \geq 1$ we have \[ N_{\rm aff}(X,B) \leq d N_{\rm aff}(X', c_n d^{e_n} B), \] where $c_n, e_n$ are constants depending only on $n$.
%  
%  For $m=1$, we can even achieve \[ N_{\rm aff}(X,B) \leq N_{\rm aff}(X', c_n d^{e_n} B) + d^2 .\]
%\end{prop}

%\begin{prop}\label{prop:projective_reduction_projection}
%  Given an integral projective variety $X$ in $\PP^n$ of dimension $m$ and degree $d$, there exists an integral projective variety $X'$ in $\PP^{m+1}$ birational to $X$, also of degree $d$, such that for any $B \geq 1$ we have \[ N(X,B) \leq d N(X', c_n d^{e_n} B), \] where $c_n, e_n$ are constants depending only on $n$.
%  
%  For $m=1$, we can even achieve \[ N(X,B) \leq N(X', c_n d^{e_n} B) + d^2 .\]
%\end{prop}

\begin{proof}[Proof of Proposition \ref{prop:projective_reduction_projection}]
  Let $X$ be a geometrically integral projective variety in $\PP^n$ of dimension $m$ and degree $d$, where we may assume that $n>m+1$.
  We consider a projection $p_{\Lambda, \Gamma}$ as in Lemma \ref{lem:projection}. By dropping appropriately chosen coordinates, its image $X'$ can be viewed as a hypersurface in $\PP^{m+1}$, birational to $X$ and hence also of degree $d$.
  In each fibre of $p_{\Lambda, \Gamma}$ there are at most $d$ points. The height relation from Lemma \ref{lem:projection} now immediately implies
  \[  N(X,B) \leq d N(X', c_n d^{2(n-m-1)^2} B) \]
  for all $B \geq 1$. %, with the constant $c$ as in Lemma \ref{lem:projection}.
  This proves the claim for $m>1$.
  For $m=1$, consider the normalization $\tilde{X} \to X$ and compose it with the morphism $X \to X'$ induced by $p_{\Lambda, \Gamma}$ to find a resolution of singularities $\tilde{X} \to X'$. 
The latter map is one-to-one away from the singular points of $X'$, which together have no more than $(d-1)(d-2)$ preimages by~\cite[Theorem 17.7(b)]{kunz}. But then the same claims must apply to $X \to X'$, yielding the stronger bound
  \[ N(X,B) \leq N(X', c_n d^{2(n-2)^2} B) + d^2,  \]
as wanted.
\end{proof}

\begin{proof}[Proof of Proposition \ref{prop:affine_reduction_projection}]
  Let $X$ be a geometrically integral affine variety in $\AA^n$ of dimension $m$ and degree $d$, where we may assume that $m<n-1$.
  Let $Z$ be the projective closure of $X$ in $\PP^n$; we apply Lemma \ref{lem:projection} and shall argue later that we can take the $(n-m-2)$-plane $\Lambda$ to be contained in the hyperplane $\PP^{n-1}$ at infinity.
  Let $Z' \subseteq \Gamma$ be the image of $Z$ under the projection $p_{\Lambda, \Gamma}$. As above, by dropping some coordinates we can view $\Gamma$ as $\PP^{m+1} = \mathbb{A}^{m+1} \sqcup \PP^m$ where $p_{\Lambda, \Gamma}(\PP^{n-1} \setminus \Lambda)$ corresponds to $\PP^m$. In particular $p_{\Lambda, \Gamma}$ maps $X$ to the affine part $X_0' = Z' \cap \AA^{m+1}$ of $Z'$.
  
  Consider $P_1, P_2, \ldots, P_{n-m-1}$ and $L_1, L_2, \ldots, L_{n-m-1}$ as in the proof of Lemma~\ref{lem:projection}. Let $P \in X$ be a point having integer coordinates; when considered as a projective point of $Z$ its coordinate at infinity is $1$. Since the coordinates at infinity of the $P_i$'s are $0$, the projection formula~\eqref{projection_formula} shows that $p_{\Lambda, \Gamma}(P) \in Z'$ admits integer coordinates such that the coordinate at infinity is
  \[ L_1(P_1) L_2(P_2) \cdots L_{n-m-1}(P_{n-m-1}), \]
regardless of the choice of $P$. As a consequence, this is a multiple of the denominators appearing among the coordinates of $p_{\Lambda, \Gamma}(P)$ when viewed as an affine rational point of $X_0'$. Therefore, postcomposing with a coordinate scaling map $\AA^{m+1} \to \AA^{m+1}$, we obtain another variety $X'$ in $\AA^{m+1}$ such that every integral point $P$ of $X$ is mapped to an integral point of $X'$ whose height satisfies the same upper bound as in~\eqref{projectionlemmabound}.
%\[ (n-m)((n+1)B_1B_2)^{n-m} H(P) = c d^\frac{2(n-m)(2n-m-2)}{m+2} H(P) \]
%for some constant $c=c(n)$. 
  All fibres of this map $X \to X'$ have at most $d$ points, and in the case of curves the map is even one-to-one away from the singular points on $X'$.
 So we can conclude as in the proof of Proposition \ref{prop:projective_reduction_projection}. 
  
  It remains to argue why we can take $\Lambda$ in the hyperplane at infinity. We first claim that the ``good set'' $G_Z$ from Lemma~\ref{empty} has a non-empty intersection with the Grassmannian parametrizing $(n-m-2)$-planes $\Lambda$ contained in $\PP^{n-1}$.
Indeed, it is apparent that the generic such $\Lambda$ does not intersect the $(m-1)$-dimensional set $Z \cap \PP^{n-1}$ and hence satisfies $\Lambda \cap Z = \emptyset$.
Furthermore, the argument from~\cite[p.\,224]{Harris} showing that generic projections are birational leaves enough freedom to draw the same conclusion when restricting to projections from planes at infinity. More precisely, if $m = n-2$ then it suffices to project from a point outside the cone spanned by $Z$ and some random point $q \in Z$. Since this cone is irreducible of dimension at most $m+1 = n-1$ and since $Z \not \subseteq \PP^{n-1}$, the generic point at infinity indeed meets this requirement. If $m < n-2$ then the desired conclusion follows by applying the foregoing argument to 
$n - m -1$ successive projections from points.

So we can redo the proof of Lemma~\ref{Existence} starting from a polynomial $F$ of degree less than $(m+1)^2d^2$ which vanishes on the complement of $G_X$ but which does not vanish identically on the Grassmannian of $(n-m-2)$-planes that are contained in the hyperplane at infinity; we just argued that such an $F$ exists. Then one can proceed with the same polynomial $Q$ as before, but with zeroes substituted for the variables $x_{10}, x_{20}, \ldots, x_{n-m-1,0}$.
\end{proof}

\section{Lower bounds}\label{sec7} 

%\subsection{Lower bounds} 
We conclude with some lower bounds showing that one cannot make the dependence on $d$ sub-polynomial. Our main auxiliary tool is the following lemma.

\begin{lem} \label{lem:lowerbound}
For each pair of integers $d \geq 1$, $n \geq 2$ there exists an absolutely irreducible degree $d$ polynomial $f \in \QQ[x_1, x_2, \ldots, x_n]$
which vanishes at all integral points
$(r_1, r_2, \ldots, r_n)$ for which $|r_i| \leq \lfloor(d-1)/2n \rfloor$ for all $i$.
\end{lem}

\begin{proof}
The lemma is immediate if $d=1$, so we can assume that $d \geq 2$.
We claim that there exists a polynomial
\[ x_1^d + x_2^d + \ldots + x_{n-1}^d + x_n^{d-1} + \sum_{0 \leq i_1, \ldots, i_n \leq \lfloor(d-1)/n \rfloor} a_{i_1, \ldots, i_n} x_1^{i_1} x_2^{i_2} \cdots x_n^{i_n} \]
which vanishes simultaneously at the integral points $(r_1, r_2, \ldots, r_n)$ satisfying
\[ \lfloor(d-1)/2n \rfloor - \lfloor(d-1)/n \rfloor  \leq r_i \leq \lfloor(d-1)/2n \rfloor \]
for all $i$. From this the lemma follows, because indeed $\lfloor(d-1)/2n \rfloor - \lfloor(d-1)/n \rfloor \leq -  \lfloor(d-1)/2n \rfloor$ and because the polynomial
is absolutely irreducible, as its Newton polytope is indecomposable; see e.g.~\cite[Theorem~4.11]{gao}.
To verify the claim, note that
every point $(r_1,r_2, \ldots, r_n)$ imposes a linear condition on the coefficients $a_{i_1, \ldots, i_n}$, together resulting
in a linear system of $(\lfloor(d-1)/n \rfloor + 1)^n$ equations
in the same number of unknowns.
It suffices to see that the matrix corresponding to its linear part is non-singular. But this matrix
is the $n$th Kronecker power of the Vandermonde matrix $(r^i)_{r,i}$ where $r$ and $i$ range over
\[ \{ \lfloor(d-1)/2n \rfloor - \lfloor(d-1)/n \rfloor , \ldots,  \lfloor(d-1)/2n \rfloor \} \quad \text{resp.} \quad \{ 0 , \ldots,  \lfloor(d-1)/n \rfloor \}. \]
Therefore its determinant is a power of the determinant of this Vandermonde matrix, from which the desired conclusion follows.
\end{proof}

%\begin{cor}
%For each integer $d>0$ there is an integral projective curve $X\subseteq \PP^2$ of degree $d$ and an integer $B\geq 1$ such that
%$$
%\frac{1}{5}d^2 B^{2/d}  \leq N(X,B).
%$$
%\end{cor}
\begin{proof}[Proof of Proposition~\ref{prop:W1.1:opt}]
If $d = 1, 2$ then we let $X$ be a line resp.\ conic through a coordinate point, so that we can take $B = 1$. If $d \geq 3$ then we
consider the affine curve defined by the polynomial $f$ from the proof of the foregoing lemma for $n=2$. Let $X$ be its projective closure, which has an extra height $1$ point at infinity.
With
$B = \lfloor(d-1)/2 \rfloor - \lfloor(d-1)/4 \rfloor$ one observes that
\[ N(X, B) \geq
(\lfloor(d-1)/2 \rfloor + 1)^2 + 1 \geq d^2 / 4 = d^2/5 \cdot 5/4 \geq d^2/5 \cdot B^{2/d}. \qedhere \]
\end{proof}

Note that using the same $f$ and $B$ one also finds that
\[ N_{\rm aff}(f,B) \geq (\lfloor(d-1)/2 \rfloor + 1)^2 \geq \frac{d^2}{4 \log d} B^{1/d} \log B \]
for all $d \geq 3$, confirming our claim that, in the statement of Theorem~\ref{thm:Bombieri-Pila}, it is impossible to replace the quartic dependence on $d$ by any expression which is $o(d^2/ \log d)$. In arbitrary dimension, the same reasoning 
shows that there exists a positive constant $c = c(n)$ such that for all integers $d>0$ we can find an absolutely irreducible degree $d$ polynomial $f \in \QQ[x_1, x_2, \ldots, x_n]$ along with an integer $B\geq 1$ such that
$$
  N_{\rm aff}(f,B) \geq cd^2B^{n-2} \quad \text{and} \quad N(X,B) \geq c d B^{\dim X},
$$
where $X \subseteq \PP^n_\QQ$ denotes the integral degree $d$ hypersurface defined by the homogenization of $f$.
This shows that Theorems~\ref{thm:dcgdegree} and~\ref{thm:0.4} cannot hold with $e < 1$ resp.\ $e < 2$.

\bibliographystyle{amsplain}
\bibliography{anbib}
%\appendix{\textbf{Appendix: Lock File with the log of names and dates}}
%\input{Lock.tex}
\end{document}